\documentclass[10pt]{amsart}
\usepackage[utf8]{inputenc}
\usepackage{fontenc}
\usepackage{amsfonts}
\usepackage{amssymb}
\usepackage{amsmath}
\usepackage{amsthm}
\usepackage{enumerate}
\usepackage{mathrsfs}
\usepackage{mathtools}

\usepackage{hyperref}
\usepackage{tikz}
\newcommand{\rs}{\mathord{\upharpoonright}}

\newcommand{\e}{\varepsilon}

\newcommand{\NN}{\mathbb{N}}
\newcommand{\zet}{\mathbb{Z}}

\newcommand{\er}{\mathbb{R}}
\newcommand{\ce}{\mathbb{C}}

\newcommand{\bb}{\mathfrak{b}}

\newcommand{\Cstar}{\mathrm{C^*}}
\newcommand{\cstar}{$\mathrm{C}^*$}

\newtheorem{theorem}{Theorem}[section]
\newtheorem*{theorem*}{Theorem}

\newtheorem{proposition}[theorem]{Proposition}
\newtheorem*{proposition*}{Proposition}
\newtheorem{lemma}[theorem]{Lemma}
\newtheorem*{lemma*}{Lemma}
\newtheorem{corollary}[theorem]{Corollary}
\newtheorem*{corollary*}{Corollar}

\newtheorem*{fact*}{Fact}
\theoremstyle{definition}
\newtheorem{defin}[theorem]{Definition}
\newtheorem*{defin*}{Definition}
\newtheorem{claim}[theorem]{Claim}
\newtheorem*{claim*}{Claim}

\newtheorem*{conjecture*}{Conjecture}

\newtheorem{theoremi}{Theorem}

\theoremstyle{remark}

\newtheorem*{example*}{Example}
\newtheorem{remark}[theorem]{Remark}
\newtheorem*{remark*}{Remark}

\newtheorem*{note*}{Note}

\newtheorem*{question*}{Question}

\newcommand{\norm}[1]{\left\lVert #1 \right\rVert}

\DeclareMathOperator{\diag}{diag}
\DeclareMathOperator{\coker}{coker}
\DeclareMathOperator{\Span}{span}

\DeclareMathOperator{\supp}{supp}
\DeclareMathOperator{\id}{id}

\DeclareMathOperator{\rank}{rank}

\DeclareMathOperator{\Ad}{Ad}

\DeclareMathOperator{\Obj}{Obj}
\DeclareMathOperator{\Mor}{Mor}

\newcommand{\Fraisse}{Fra\"{i}ss\'e }

\title{Stably projectionless Fra\"iss\'e limits}

\author[B. Jacelon]{Bhishan Jacelon}
\address[B. Jacelon]{
Institute of Mathematics of the Czech Academy of Sciences\\ \v{Z}itn\'{a} 25\\115 67 Prague 1\\Czech Republic}
\email{bjacelon@gmail.com}

\author[A. Vignati]{Alessandro Vignati}
\address[A. Vignati]{
Universit\'e de Paris - Institut de Math\'ematiques de Jussieu - Paris Rive Gauche (IMJ-PRG)\\
B\^atiment Sophie Germain\\
8 Place Aur\'elie Nemours\\ Paris, 75013, France}
\email{ale.vignati@gmail.com}
\urladdr{http://www.automorph.net/avignati}

\thanks{}
\begin{document}

\maketitle
\setcounter{secnumdepth}{4}
\setcounter{tocdepth}{2}
%\tableofcontents

\date{\today}%
% ----------------------------------------------------------------
\begin{abstract}
We realise the algebra $\mathcal W$, the algebra $\mathcal Z_0$ and the algebras $\mathcal Z_0\otimes A$, where $A$ is a unital separable UHF algebra as \Fraisse limits of suitable classes of structures. In doing so, we show that such algebras are generic objects without the use of any classification result.
 \end{abstract}
 
\maketitle
\section{Introduction}

The notions of \Fraisse classes and \Fraisse limits were originally introduced by \Fraisse in \cite{Fraisse}, as a method to construct countable homogeneous structures. Since then, \Fraisse theory has become an influential area of mathematics at the crossroads of combinatorics and model theory. Broadly speaking, \Fraisse theory studies the correspondence between homogeneous structures and properties of the classes of their finitely generated substructures. In the discrete setting, given a countable structure, its age is the collection of its finitely generated substructures. Ages of homogeneous structures are precisely \Fraisse classes. Conversely, given a \Fraisse class $\mathcal K$, one constructs a countable homogeneous structure with the given class as its age. This
structure is the \Fraisse limit of the class. It is unique up to isomorphism and is often referred to as the \emph{generic} structure one gets from $\mathcal K$.

Many interesting objets in group theory, graph theory, and topology were identified as \Fraisse limits (see for example \cite[Chapter 7]{Hodges1993} and \cite{IrwSol.Pseudo}). Connections with Ramsey theory and topological dynamics, leading to the study of extreme amenability of the automorphism group of \Fraisse limits, were exploited in \cite{KPT}.

After an early approach in \cite{Schoretsanitis}, \Fraisse theory for continuous structures was developed systematically in \cite{BY:Fraissee}, where it was notably shown that the Urysohn space is a \Fraisse limit of the class of finite metric spaces. Another object of pivotal importance recognised as a \Fraisse limit is the Gurarij Banach space, see \cite{KubisSol.Gurarij}. The `continuous' correspondent of the main result of \cite{KPT} was proved in \cite[Theorem 3.10]{melleray2014extremely}.

\Fraisse theory was brought to the setting of \cstar-algebras in \cite{EFH:Fraisse}. When studying such objects, one often has to consider classes which are not closed under substructures, the reason for this being that the class of finitely generated substructures of a given \cstar-algebra is often quite large and intractable (conjecturally, all simple and separable C*-algebras are singly generated, see e.g., \cite{ThielWinter}). This phaenomenon translates to the \Fraisse class not having the Hereditary Property. Yet, we consider classes made of reasonably `small' and `tractable' algebras, which will form a `skeleton' in the age of the \Fraisse limit. The price one has to pay in this case is that one only obtains that the \Fraisse limit is homogeneous only for certain maps from the building blocks into the \Fraisse limit.

For us, a \Fraisse class is a category $\mathcal K$ with objects $\Obj_{\mathcal K}$ and morphisms $\Mor_{\mathcal K}$. The objects of $\mathcal K$ are finitely generated metric $\mathcal L$-structures, and the morphisms are $\mathcal L$-embeddings, where $\mathcal L$ is a language for metric structures in continuous model theory. We ask for $\mathcal K$ to satisfy certain combinatorial properties (see \S\ref{SS.Fraisse} for the specifics). The most important among these are the Joint Embedding Property (JEP) and the Near Amalgamation Property (NAP). The JEP asks that every two objects of $\mathcal K$ embed into a third one via maps in $\Mor_\mathcal K$, while our amalgamation property NAP asks that objects in $\mathcal K$ are local amalgamation bases, at least when diagrams are restricted to $\mathcal K$: $\mathcal K$ satisfies the NAP if whenever we are given objects $A$, $B$, and $C$, morphisms $\varphi_1\colon A\to B$ and $\varphi_2\colon A\to C$, a finite set $F\subset A$ and $\e>0$, then we can find an object $D$ and morphisms $\psi_1\colon B\to D$ and $\psi_2\colon C\to D$ such that 
\[
d(\psi_1\circ\varphi_1(a),\psi_2\circ\varphi_2(a))<\e \text{ for all } a\in F,
\]
$d$ being the metric on $D$. The NAP is usually the more technical property to prove, but the interesting one as it gives homogeneity properties to the generic inductive limit of a \Fraisse class, its \Fraisse limit. In the setting of \cstar-algebras, the JEP takes the role of `local existence', while the NAP takes that of `local uniqueness'.

Notably, the authors of \cite{EFH:Fraisse} showed that the Jiang--Su algebra $\mathcal Z$ and the UHF algebras of infinite type are \Fraisse limits of suitable \Fraisse classes. Masumoto, in \cite{Masumoto.FraisseZ,Masumoto.Fraisse2} and \cite{Masumoto.Real}, obtained the same results with a `by hand' approach, not relying on any classification theory. Ghasemi, in \cite{Ghasemi.Fraisse}, further analysed the connections between \Fraisse theory and strongly self-absorbing \cstar-algebras to give a self-contained and rather elementary proof for the well known fact that $\mathcal Z$ is strongly self-absorbing.

Both Jiang and Su's $\mathcal Z$ and UHF algebras of infinite type are examples of \cstar-algebras of fundamental importance in the classification programme. In particular, $\mathcal Z$ can be viewed as the (stably finite) infinite dimensional version of the complex numbers $\mathbb C$; $\mathcal Z$ plays a pivotal role in the classification of infinite-dimensional simple separable nuclear \cstar-algebras, where tensorial absorption of $\mathcal Z$ is proved to be equivalent to a finite-dimensionality condition (\cite{CETWW.Class} and \cite{CE.ClassNonunit}). 

In this paper, we focus on $\mathcal Z$'s nonunital twins, the algebras $\mathcal W$ and $\mathcal Z_0$. These algebras are \emph{the} simple, infinite-dimensional, amenable, stably projectionless algebras with unique (bounded) trace whose $K$-theory is as simple as possible. $\mathcal W$ has trivial $K$-theory, while the $K$-theory of $\mathcal Z_0$ equals that of $\mathbb C$. The algebra $\mathcal W$ was defined by the first author in \cite{Jacelon:W}, following the pioneering work of Razak \cite{Razak}, who identified a class of nonunital separable stably finite and stably projectionless simple nuclear \cstar-algebras with trivial $K$-theory, and classified them using their tracial information. $\mathcal W$ was defined to mimic the properties of $\mathcal Z$, namely in the attempt of defining a `strongly self-absorbing' nonunital \cstar-algebra. Strong self-absorption is a property asserting that $A\cong A\otimes A$ in quite a strong way (see \cite{Toms:2007uq}). That $\mathcal W\cong\mathcal W\otimes\mathcal W$ was proved only recently using the classification tools of \cite{EllGongLinNiu.KK} (see also \cite{Nawata:2020vt}). 

The algebra $\mathcal Z_0$ may be seen as yet another nonunital version of $\mathcal Z$. For classification purposes, $\mathcal Z_0$ plays for nonunital algebras the role $\mathcal Z$ does for unital ones (compare, for example, \cite[Theorem 1.2]{GongLin.Class2} with \cite[Corollary E]{TWW.QD} and the main result of \cite{GongLinNiu}). 

We identify these two important objects as \Fraisse limits.

\begin{theoremi}\label{thmi:main}
The algebras $\mathcal W$ and $\mathcal Z_0$ are \Fraisse limits of suitable classes. Moreover, if $A$ is a UHF algebra of infinite type, the algebra $\mathcal Z_0\otimes A$ is a \Fraisse limit.
\end{theoremi}

Our approach does not use any classification result. In particular we show, for the first time, that $\mathcal W$ has the properties of a generic object without the use of any classification tool. 

Our \Fraisse classes consist of Razak blocks and their generalised versions, together with a specified faithful diffuse trace; the morphisms we are interested in are trace preserving $^*$-homomorphisms. (Generalised) Razak blocks are subalgebras of algebras of the form $C([0,1],M_n)$ defined by certain boundary conditions at $0$ and $1$ (see \S\ref{subsection:blocks}). For $\mathcal Z_0$ and algebras of the form $\mathcal Z_0\otimes A$ where $A$ is a unital UHF algebra of infinite type, we also account for $K$-theoretic constraints. The upshot of such analysis is twofold: first, we manage to classify (by traces, and $K$-theoretic information) embeddings of Razak blocks (resp, generalised Razak blocks) into $\mathcal W$ (resp, $\mathcal Z_0$) without the use of any classification theory. Second, we start a promising model theoretic analysis of two pivotal objects such as $\mathcal W$ and $\mathcal Z_0$. Moreover, this is the first time the algebra $\mathcal Z_0$ has been explicitly expressed as an inductive limit of subhomogeneous building blocks. To the best of our knowledge, a similar approach was already present in unpublished work of Santiago, but the best picture of $\mathcal Z_0$ so far available in written form was the one sketched out in \cite[\S7]{GongLin.Class2}, where $\mathcal Z_0$ is realised as a limit of subalgebras of $C([0,1],\mathcal Q\otimes\mathcal Q)$, $\mathcal Q$ being the universal UHF algebra.

The key part of the proof of Theorem~\ref{thmi:main} is proving NAP for our classes, that is, local uniqueness. For this, we study certain distances between trace preserving $^*$-homomorphisms of (generalised) Razak blocks, measures, and sets, and how these interplay (see \S\ref{S.distances}). Through the notion of diameter (\S\ref{SS.Diameter}), we measure the amplitude of $^*$-homomorphisms obtained from continuous maps $[0,1]\to[0,1]$, and we show that obtaining maps with small diameters suffices for our scope. In particular, the idea is that maps of small diameter that pull back the same trace are pointwise unitarily close (this is what \S\ref{S.distances} amounts to). We then, in \S\ref{S.proof}, use a combinatorial argument due to Robert  (\cite[\S5]{Robert:2010qy}) to generalise a result of Thomsen (\cite{Thomsen:1992qf}), thereby obtaining a continuous conjugating unitary in the unitisation of a (generalised) Razak block.

The paper is structured as follows: \S\ref{S.Prelim} contains preliminaries; there, we introduce our classes of objects and their maps. In \S\ref{S.diagonal} we introduce diagonal maps, and show their basic properties; by proving the existence of diagonal maps between (generalised) Razak blocks, we show that our classes have the JEP. In \S\ref{S.distances} we introduce several distances between $^*$-homomorphisms, measures, and continuous maps $[0,1]\to[0,1]$, and we relate them to each other. Finally, \S\ref{S.proof} uses the previous sections to prove the NAP for our classes of interest, and contains the proof of our main result. Lastly, in Appendix~\ref{App.Maps} we deal with the issue of what kind of maps one obtains homogeneity for, hereby answering a question of Masumoto from \cite{Masumoto.Fraisse2}.

\subsection*{Acknowledgements} The first ideas that led to this work were developed when both authors were at the Fields Institute in Toronto in 2017. This work continued with the first author's visit to the second one at the Institut de Math\'ematiques de Jussieu-Paris Rive Gauche (IMJ-PRG), in 2018, funded by the second author's PRESTIGE grant. The authors would like to thank both institutions. The first author is supported by the GA\v{C}R project 20-17488Y and RVO: 67985840. The second author is supported by an ANR grant (ANR-17-CE40-0026) and an Emergence en Recherche IdeX grant from the Universit\'e de Paris.

\section{Preliminaries}\label{S.Prelim}
\subsection{\Fraisse classes}\label{SS.Fraisse}
We work in the setting of continuous model theory for metric structures and fix a language for metric structures $\mathcal L$. In our applications, we will work in the language of \cstar-algebras $\mathcal L_{\Cstar}$ (see \cite{FHS.II}, or \cite{bourbaki}). An $\mathcal L$-class $\mathcal K$ consists of 
\begin{itemize}
\item $\Obj_{\mathcal K}$, the objects of $\mathcal K$, which are finitely generated $\mathcal L$-structures, and,
\item for every $A,B\in\Obj_{\mathcal K}$, a set $\Mor_{\mathcal K}(A,B)$ of $\mathcal L$-embeddings, the morphisms of $\mathcal K$. 
\end{itemize}

\begin{defin}
An $\mathcal L$-class $\mathcal K$ is said to have 
 
\begin{itemize}
 \item the \emph{joint embedding property} (JEP) if for all $A_1,A_2\in\Obj_{\mathcal K}$ there is $B\in\Obj_\mathcal K$ and $\alpha_i\in\Mor_{\mathcal K}(A_i,B)$; 
 \item the \emph{near amalgamation property} (NAP) if for all $A, B_1,B_2\in\Obj_{\mathcal K}$, for each finite  $F\subset A$, $\varepsilon>0$ and $\alpha_i\in\Mor_{\mathcal K}(A,B_i)$ there is $C\in\Obj_{\mathcal K}$ and morphisms $\beta_i\in\Mor_{\mathcal K}(B_i,C)$ with 
 \[
 d(\beta_1\circ\alpha_1(f),\beta_2\circ\alpha_2(f))<\varepsilon, \,\, f\in F.
 \]
 \end{itemize}
 Let $\mathcal K_n$ be the set formed by pairs $(A,\bar a)$ where $A\in\Obj_\mathcal K$ and $\bar a\in A^n$ generates $A$ as an $\mathcal L$-structure. For $(A_1,\bar a_1),(A_2,\bar a_2)\in\mathcal K_n$ define
 \[
 d^\mathcal K((A_1,\bar a_1),(A_2,\bar a_2))=\inf_{B\in\Obj_{\mathcal K}, \alpha_i\in\Mor(A_i,B)}d(\alpha_1(\bar a_1),\alpha_2(\bar a_2)).
 \]
 If $\mathcal K$ has JEP and NAP, $d^{\mathcal K}$ is a pseudo-metric. We say that $\mathcal K$ has
 \begin{itemize}
 \item the \emph{weak Polish property} (WPP) if each $\mathcal K_n$ is separable in the topology generated by $d^\mathcal K$;
 \item the \emph{Cauchy continuity property} (CCP) if for all $n,m\in\NN$ and $n$-ary $\mathcal L$-predicates $P$ and $m$-ary $\mathcal L$-functions $f$ we have that 
\[
 (A,\bar a,\bar b)\mapsto P^A(\bar a)\,\,\,\,\text{ and } (A,\bar a,\bar b)\mapsto (A,\bar a,\bar b,f^A(\bar a)) 
\]
send Cauchy sequences in $\mathcal K_{n+m}$ to Cauchy sequences in $\er$ and $\mathcal K_{n+m+1}$ respectively.
 \end{itemize}
\end{defin}

\begin{remark}\label{rem:CCP}
If $\mathcal L$ is the language of tracial \cstar-algebras then CCP is automatic: all functions and predicates in the language are $1$-Lipschitz.
\end{remark}

\begin{defin}
Let $\mathcal L$ be a separable language of metric structures and $\mathcal K$ be an $\mathcal L$-class. If $\mathcal K$ satisfies JEP, NAP, WPP and CCP, $\mathcal K$ is called a \emph{\Fraisse class}.
\end{defin}

If $\mathcal K$ is an $\mathcal L$-class, $A_i\in\Obj_{\mathcal K}$, $i\in\NN$, and $\varphi_i\in\Mor_{\mathcal K}(A_i,A_{i+1})$, the $\mathcal L$-structure 
\[
M=\lim (A_i,\varphi_i)
\] 
is called a $\mathcal K$-structure. We call $M$
\begin{itemize}
\item $\mathcal K$-\emph{universal} if every $A\in\Obj_{\mathcal K}$ can be $\mathcal K$-admissibly embedded in $M$ and
\item \emph{approximately }$\mathcal K$-\emph{homogeneous} if for every $A\in\Obj_{\mathcal K}$, $\varepsilon>0$, a finite $F\subset A$ and two $\mathcal K$-admissible embeddings $\alpha_1,\alpha_2\colon A\to M$ there is a $\mathcal K$-admissible isomorphism $\varphi\colon M\to M$ such that 
\[
d(\varphi\circ\alpha_2(f),\alpha_1(f))<\varepsilon, \,\, f\in F.
\]
\end{itemize}
\begin{defin}\label{D.Fraisse}
Let $\mathcal L$ be a separable language of metric structures and let $\mathcal K$ be a \Fraisse class. A $\mathcal K$-structure which is $\mathcal K$-universal and approximately $\mathcal K$-homogeneous is called a \emph{\Fraisse limit} of $\mathcal K$. 
\end{defin}

\begin{remark}\label{remark:admissible}
The definition of $\mathcal K$-admissible map above is quite technical. The need of this technical restriction on `allowed' morphisms from objects in $\mathcal K$ to $\mathcal K$-structures is due to the absence of the Hereditary Property. This absence is usually irrelevant in the discrete setting (for instance the classes considered by Irwin and Solecki in \cite{IrwSol.Pseudo} do not have the Hereditary Property), but it creates technical issues in the continuous setting (see, e.g., the introduction of \cite{Masumoto.Fraisse2}).

While not every embedding of $\mathcal K$-structures is $\mathcal K$-admissible according to the technical definition of Masumoto (\cite[Definition 3.1 (5)]{Masumoto.Fraisse2}), the class of $\mathcal K$-admissible maps is rich enough, and it has the following preservation properties:
\sloppy
\begin{itemize}
\item if $A,B\in\Obj_\mathcal K$ and $\varphi\colon A\to B$ is in $\Mor_\mathcal K$, then $\varphi$ is $\mathcal K$-admissible;
\item  if $A_i\in\Obj_\mathcal K$ and $\varphi_i\colon A_i\to A_{i+1}$ are elements of $\Mor_\mathcal K$ then $\varphi_{i,\infty}\colon A_i\to\lim(A_i,\varphi_i)$ given by $a\mapsto\lim_{j>i}\varphi_{ij}(a)$, where $\varphi_{ij}\colon A_i\to A_j$ equals $\varphi_j\circ\cdots\circ\varphi_{i}$, is $\mathcal K$-admissible;
\item if $A\in\Obj_\mathcal K$ and $B=\lim (B_i,\varphi_i)$ is a $\mathcal K$-structure, where $B_i\in\Obj_\mathcal K$ and $\varphi_i\in\Mor_\mathcal K$, then an $\mathcal L$-embedding $\psi\colon A\to B$ such that for all finite $F\subset A$ and $\varepsilon>0$ there are $n$ and $\varphi\colon A\to B_n$ with $\varphi\in\Mor_\mathcal K$ for which 
\[
\norm{\psi(a)-\varphi_{i,\infty}\circ\varphi(a)}<\varepsilon, \, \, a\in F,
\]
is $\mathcal K$-admissible;
\item if $A=\lim (A_i,\varphi_i)$ and $B=\lim(B_i,\psi_i)$ are $\mathcal K$-structures, and $\rho$ is an $\mathcal L$-embedding $\rho\colon A\to B$ such that for all $i$, for each finite $F\subset A_i$ and $\varepsilon>0$, there is $j$ and a $\tilde\rho\in\Mor_{\mathcal K}$ with $\tilde\rho\colon A_i\to B_j$ such that 
\[
\norm{\tilde\psi_j\circ\tilde\rho(a)-\rho\circ\varphi_{i,\infty}(a)}<\varepsilon, \,\,\, a\in F,
\]
then $\rho$ is $\mathcal K$-admissible.
\item Let $A,B$ be $\mathcal K$-structures and $\varphi\colon A\to B$ be an $\mathcal L$-embedding. Suppose that $A=\lim (A_i,\varphi_i)$. If $\varphi\restriction A_i\colon A_i\to B$ is $\mathcal K$-admissible for all sufficiently large $i$, then so is $\varphi$. 
\end{itemize}
\fussy

We will return to $\mathcal K$-admissible morphisms in Appendix \ref{App.Maps}, where we show that for our classes of interest all maps involved are admissible.
\end{remark}
\Fraisse limits exist, and they are unique:

 \begin{theorem}[{\cite[Theorem 3.15]{Masumoto.Fraisse2}}]\label{thm:Fraisse}
Let $\mathcal L$ be a separable language of metric structures and $\mathcal K$ be an $\mathcal L$-class. Then $\mathcal K$ satisfies JEP, NAP, WPP and CCP if and only if there exists a \Fraisse limit of $\mathcal K$. Such a limit is unique up to $\mathcal K$-admissible isomorphism.
\end{theorem}
The following is a simplification of \cite[Proposition 3.19]{Masumoto.Fraisse2}.
\begin{theorem}\label{thm:generic}
Let $\mathcal K$ be a \Fraisse class and $M=\lim(A_i,\varphi_i)$, where $A_i\in\Obj_\mathcal K$ and $\varphi_i\in\Mor_\mathcal K(A_i,A_{i+1})$. For $i<j$, let $\varphi_{i,j}=\varphi_{j-1}\circ\cdots\circ\varphi_i\colon A_i\to A_j$. Suppose that
\begin{itemize}
\item for every $C\in\mathcal K$ there is $i$ and $\varphi\in\Mor_{\mathcal K}(C,A_i)$, and
\item for all $i$, for each finite $F\subset A_i$, $\varepsilon>0$, $C\in\Obj_{\mathcal K}$ and $\psi\in\Mor_{\mathcal K}(A_i,C)$ there is $k$ and $\eta\in\Mor_{\mathcal K}(C,A_k)$ such that
\[
\norm{\eta\circ\psi(a)-\varphi_{i,k}(a)}<\varepsilon, \,\, a\in F.
\]
\end{itemize}
Then $M$ is the \Fraisse limit of $\mathcal K$.
\end{theorem}
A sequence $(A_i,\varphi_i)$ witnessing Theorem~\ref{thm:generic} is called \emph{generic}.

\subsection{The building blocks, their traces, and their representations} \label{subsection:blocks}
We fix some notation. If $k>0$, we denote by $M_k$ the algebra of $k\times k$-valued complex matrices. Our norm will always denote the $2$-norm, which makes $M_k$ a \cstar-algebra. $0_k$ and $1_k$ denote the $0$ matrix and the identity in $M_k$, respectively. If $a\in M_k$ and $b\in M_{k'}$, we denote by $\diag(a,b)$ the matrix
$
\begin{pmatrix}a&0\\0&b\end{pmatrix} \in M_{k+k'}
$.
This definition extends inductively. We often shorten notation, and write $\diag(\underbrace{a}_n)$ for $\diag(\underbrace{a,\ldots,a}_n)$, the matrix which has $n$-copies of $a$ on the diagonal.

If $n,k\in\NN$, let
\begin{eqnarray*}
A_{n,k}=\{f\in C([0,1],M_{nk})\mid \exists a\in M_k(f(0)=&\diag(\underbrace{a}_\text{n}),\\ f(1)=&\diag(\underbrace{a}_\text{n-1},0_k))\}
\end{eqnarray*}
and
\begin{eqnarray*}
B_{n,k}=\{f\in C([0,1],M_{2nk})\mid \exists a,b\in M_k (f(0)=&\diag(\underbrace{a}_\text{n}, \underbrace{b}_\text{n}),\\ f(1)=&\diag(\underbrace{a}_\text{n-1},0_k, \underbrace{b}_\text{n-1},0_k))\}
\end{eqnarray*}
These algebras are known as \emph{Razak blocks} (the $A_{n,k}$'s), and \emph{generalised Razak blocks} (the $B_{n,k}$'s). 
If $f\in A_{n,k}$, we denote by $a_f$ the element of $M_k$ such that 
\[
f(0)=\diag(\underbrace{a_f}_n).
\]
 If $f\in B_{n,k}$, we denote by $a_f$ and $b_f$ the elements of $M_k$ such that 
 \[
 f(0)=\diag(\underbrace{a_f}_n, \underbrace{b_f}_n).
\]
 \begin{proposition}\label{prop:Ktheory}
Let $n,k\in\NN$. Then
\begin{enumerate}
\item\label{noprojc1} (generalised) Razak blocks are stably projectionless, but every proper quotient of them has a nonzero projection;
\item\label{noprojc2} nonzero $^*$-homomorphisms between (generalised) Razak blocks are injective;
\item\label{noprojc3} $K_*(A_{n,k})=0$, $K_0(B_{n,k})\cong\zet$ and $K_1(B_{n,k})=0$.
\end{enumerate}
\end{proposition}
\begin{proof}
\eqref{noprojc2} follows from \eqref{noprojc1}, which we prove for Razak blocks, leaving the generalised case as an (easy) exercise. (1) is truly a `counting multiplicity' argument: If $f$ is a projection in $A_{n,k}$, so is $a_f$. Since $[0,1]$ is connected, the rank of $f(0)$, which equals $n\cdot\rank(a_f)$, is equal to the rank of $f(1)$, which equals $(n-1)\cdot\rank(a_f)$. Hence $\rank(a_f)=0$, and so $\rank(f(t))=0$ for all $t$, which implies that $f=0$.

If $\mathcal I$ is an ideal in $A_{n,k}$, then there is a closed $C\subseteq[0,1]$ such that $\mathcal I=\{f\mid f\rs C=0\}$. If $\mathcal I$ is nontrivial, $[0,1]\setminus C$ is nonempty. Hence there is $t\in (0,1)$ and $\varepsilon>0$ such that $(t-\varepsilon,t+\varepsilon)\cap C=\emptyset$. Any function in $A_{n,k}$ which is the identity on $[0,t-\varepsilon]$ and is a projection of rank $(n-1)k$ on $[t+\varepsilon,1]$ gives a projection in the quotient.

\eqref{noprojc3}: The $K$-theory can be computed by realising $A_{n,k}$ and $B_{n,k}$ as one-dimensional NCCW complexes (also called \emph{point-line} or \emph{Elliott-Thomsen} algebras). That is, they are pullbacks of the form
\[
A(E,F,\alpha_0,\alpha_1) = \{(f,g) \in C([0,1],F)\oplus E \mid f(0) = \alpha_0(g), f(1) = \alpha_1(g)\}
\]
for finite-dimensional algebras $E$ and $F$ and $^*$-homomorphisms $\alpha_0,\alpha_1:E\to F$. For such an algebra $A$, if $K_0(\alpha_i)$ denote the induced homomorphisms $K_0(E)\to K_0(F)$, one has $K_1(A) \cong \coker(K_0(\alpha_0)-K_0(\alpha_1))$ and $K_0(A) \cong \ker(K_0(\alpha_0)-K_0(\alpha_1))$ (with the ordering inherited from $K_0(E)$). (See \cite[Proposition 35]{GongLinNiu.Class1}.) For example, \[B_{n,k} \cong A(M_k\oplus M_k, M_{2nk},\id\otimes1_n\oplus\id\otimes1_n, \id\otimes1_{n-1}\oplus\id\otimes1_{n-1}).\] The map $K_0(\alpha_0)-K_0(\alpha_1)\colon\zet^2\to\zet$ is represented by the matrix $(1\:1)$. It follows that $K_1(B_{n,k})=0$ and $K_0(B_{n,k}) \cong \{(l,-l) \mid l\in\zet\} \cong \zet$ (with trivial positive cone, whence $B_{n,k}$ is stably projectionless). A similar calculation yields the $K$-theory of $A_{n,k}$.
\end{proof}

\begin{remark} \label{remark:generator}
While \cite[Proposition 3.5]{GongLinNiu.Class1} is stated for unital point-line algebras, it also holds in the nonunital case. This can be seen by unitising, which also helps to identify generators of $K_0$. For example,
\[
\tilde B_{n,k} \cong A(M_k\oplus M_k \oplus \ce, M_{2nk}, \id\otimes1_n\oplus\id\otimes1_n, \id\otimes1_{n-1}\oplus\id\otimes1_{n-1}\oplus\id\otimes1_{2k}).
\]
Then
\[
K_0(\tilde B_{n,k}) \cong \ker(\begin{pmatrix}1&1&-2k\end{pmatrix}\colon\zet^3\to\zet) =\Span_{\zet}\{(1,-1,0),(2k,0,1)\}.
\]
By definition, $K_0(B_{n,k})$ is the kernel of the map $K_0(\tilde B_{n,k}) \to \zet$ induced by the quotient map $\tilde B_{n,k} \to \ce$. So $K_0(B_{n,k})$ is generated by
\[
(1,-1,0) = (k+1,k-1,1) - (k,k,1) = [p_{n,k}]-[1_{\tilde B_{n,k}}],
\]
where $p_{n,k}\in M_2(\tilde B_{n,k})$ is a projection with
\[
p_{n,k}(1)=\diag(\underbrace{\diag (1_{k+1},0_{k-1})}_{n-1}, \diag(1_k,0_k), \underbrace{\diag(1_{k-1},0_{k+1})}_{n-1},\diag(1_k,0_k)).
\]
In the sequel, we will identify $K_0(B_{n,k})$ with $\mathbb{Z}$ via the generator $[p_{n,k}]-[1_{\tilde B_{n,k}}]$. Let $\varphi\colon B_{n,k}\to B_{n',k'}$ be a $^*$-homomorphism. By abuse of notation, we extend $\varphi$ to a unital map $M_2(\tilde B_{n,k})\to M_2(\tilde B_{n',k'})$ in a natural way, and we say that $\varphi$ has $K$-theory equal to $\ell\in\zet$, and write $K_0(\varphi)=\ell$, if 
\[
[\varphi(p_{n,k})]-[1_{\tilde B_{n',k'}}]=\ell([p_{n',k'}]-[1_{\tilde B_{n',k'}}]).
\]
\end{remark}

\subsubsection*{Representations}
If $\pi$ and $\rho$ are representations of the same \cstar-algebra, we write 
\[
\pi\sim_u\rho
\]
 if they are unitarily equivalent.
The space of unitary equivalence classes of nonzero irreducible representations of a \cstar-algebra $A$ is called the \emph{spectrum} $\hat A$ of $A$. Equipped with the `hull-kernel' topology, $\hat A$ is always locally compact (see \cite[\S3.3]{Dixmier:1964rt}) but often not Hausdorff. (Generalised) Razak blocks are sub-homogeneous, that is, all elements $[\pi]\in\hat A$ are finite dimensional. In fact, every such $\pi$ is equivalent to a point representation, either at an interior point $t\in(0,1)$ or at one of the `points at infinity' $\{\infty_i\}$. 

Specifically, if $\pi$ is a nonzero irreducible representation of $A_{n,k}$ then either
\begin{itemize}
\item $\dim\pi=nk$ and $\pi$ is unitarily equivalent to $f\mapsto f(t)$ for some $t\in (0,1)$, or
\item $\dim\pi=k$ and $\pi$ is unitarily equivalent to $f\mapsto a_f$.
\end{itemize}
Similarly, if $\pi$ is a nonzero irreducible representation of $B_{n,k}$ then either
\begin{itemize}
\item $\dim\pi=2nk$ and $\pi$ is unitarily equivalent to $f\mapsto f(t)$ for some $t\in (0,1)$,
\item $\dim\pi=k$ and $\pi$ is unitarily equivalent to $f\mapsto a_f$, or
\item $\dim\pi=k$ and $\pi$ is unitarily equivalent to $f\mapsto b_f$.
\end{itemize}
The above statements remain true after matrix amplification. The effect of adding a unit to a (generalised) Razak block $A$ is to add an extra point at infinity; this corresponds to the irreducible representation $\tilde{A}\to\ce$ that annihilates $A$.  

We write $\pi_t$ for the point representations $f\mapsto f(t)$. If $A$ is a Razak block, we write $\pi_\infty$ for the representation $f\mapsto a_f$. If $A$ is a generalised Razak block, we write $\pi_{\infty_1}$ and $\pi_{\infty_2}$ for the representations $f\mapsto a_f$ and $f\mapsto b_f$.
As finite-dimensional representations are unitarily equivalent to sums of irreducible ones, we have the following:
\sloppy
\begin{proposition}\label{prop:repr}
Let $n,k,m\in \NN$. Then
\begin{itemize}
\item If $\pi$ is an $m$-dimensional representation of $A_{n,k}$ then there are uniquely determined $s_{1},\ldots,s_{j}\in [0,1)$, and $r_{0},r_1\in\NN$, with $r_{1}<n$, such that
\[
\pi\sim_u \diag(\pi_{s_1},\ldots,\pi_{s_j},\underbrace{\pi_\infty}_{r_{1}},0_{r_{0}}).
\]
\item If $\pi$ is an $m$-dimensional representation of $B_{n,k}$ then there are uniquely determined $s_{1},\ldots,s_{j}\in [0,1)$, and $r_{0},r_{1},r_{2}\in\NN$, with $\min\{r_{1},r_{2}\}<n$, such that
\[
\pi\sim_u \diag(\pi_{s_1},\ldots,\pi_{s_j},\underbrace{\pi_{\infty_1}}_{r_{1}},\underbrace{\pi_{\infty_2}}_{r_{2}},0_{r_{0}}).
\]
\end{itemize}
The same descriptions hold after matrix amplification. The unique unital extension of $\pi$ to the unitisation is described by replacing $0$ with the unital representation onto $\ce$. \qed 
\end{proposition}
\fussy

\sloppy
As in the case of maps between generalised Razak blocks, if $\pi\colon B_{n,k}\to M_m$ is a representation, this induces a group homomorphism \[\mathbb Z\cong K_0(B_{n,k})\to K_0(M_m)\cong\mathbb Z.\] As before (see e.g., Remark~\ref{remark:generator}) we write $K_0(\pi)=\ell$, and say that the $K$-theory of $\pi$ is $\ell$, if  the generator of $K_0(B_{n,k})$ gets sent to $\ell$ many times the canonical generator of $K_0(M_m)$.
\fussy 
\begin{lemma}\label{lemma:KtheoryGen}
Let $A$ be a generalised Razak block, and $\pi\colon A\to M_m$ be a representation, and suppose that $r_{1},r_{2}$ are the numbers given by Proposition~\ref{prop:repr}. Then $K_0(\pi)=r_{1}-r_{2}$.
\end{lemma}
\begin{proof}
This follows from the definition of the generator $[p_{n,k}]-[1_{\tilde B_{n,k}}]$ of $K_0(B_{n,k})$ (see Remark~\ref{remark:generator}), and the fact that in the identification of $K_0(M_m)$ with $\mathbb{Z}$, a difference of projection classes $[q_1]-[q_2]$ corresponds to $\rank(q_1)-\rank(q_2)$.
\end{proof}

The following stable uniqueness lemma will be used in the proof of Theorem~\ref{T.gensmalldistance}.

\begin{lemma}\label{lemma:KtheoryGen2}
Let $A$ be a generalised Razak block and $\rho_1,\rho_2\colon A\to M_q$ be two representations with $K_0(\rho_1)=K_0(\rho_2)=\ell$.
Then there exists $j\in\NN$ and points $x_1,\ldots,x_j, y_1,\ldots,y_j$ in $[0,1]$ such that 
 \[
 \diag(\rho_1,\pi_{x_1},\ldots,\pi_{x_j})\sim_u \diag(\rho_2,\pi_{y_1},\ldots,\pi_{y_j}).
 \] 
 \end{lemma}
\begin{proof}
\sloppy
Without loss of generality, we can assume $\ell\geq0$ (if not, replace $\infty_1$ with $\infty_2$ in the argument below). By Proposition~\ref{prop:repr} and Lemma~\ref{lemma:KtheoryGen}, we can find natural numbers $m,m', r_{1,1},r_{2,1},r_{1,0}$ and $r_{2,0}$, and points $x_1,\ldots,x_m$ and $y_1,\ldots,y_{m'}$ in $[0,1)$ such that for all $f\in A$ we have
\[
\rho_1\sim_u\diag(\underbrace{\pi_{\infty_1}}_\ell, \underbrace{\diag(\pi_{\infty_1},\pi_{\infty_2})}_{r_{1,1}},0_{r_{1,0}},\pi_{x_1},\dots,\pi_{x_m})
\]
and 
\[
\rho_2\sim_u\diag(\underbrace{\pi_{\infty_1}}_\ell, \underbrace{\diag(\pi_{\infty_1},\pi_{\infty_2})}_{r_{2,1}},0_{r_{2,0}},\pi_{y_1},\dots,\pi_{y_{m'}}),
\]
 where $r_{1,1},r_{2,1}<n$. Without loss of generality we can assume that $r_{1,0}\geq r_{2,0}$.
\fussy

\noindent \underline{\textbf{Case 1:}} $r_{1,0}= r_{2,0}$.
As the two representations have the same dimension, and since $\pi_{\infty_1}$ and $\pi_{\infty_2}$ have dimension $k$ and $\pi_t$ has dimension $2nk$ for all $t\in [0,1]$, then 
\[
(2r_{1,1}k+r_{1,0}+\ell k)=(2r_{2,1}k+r_{2,0}+\ell k)\mod 2nk,
\]
 hence $2r_{1,1}k=2r_{2,1}k\mod2nk$. As $r_{1,1},r_{2,1}<n$, then $r_{1,1}=r_{2,1}$, and therefore $m=m'$. Set $j=m$, and let $z_i=y_i$ and $w_i=x_i$ for all $i\leq m$. Then 
 \[
 \diag(\rho_1,\pi_{z_1},\ldots,\pi_{z_j})\sim_u \diag(\rho_2,\pi_{w_1},\ldots,\pi_{w_j}).
 \] This is the thesis.

\noindent \underline{\textbf{Case 2:}} $r_{1,0}> r_{2,0}$.
Counting the size of representations as above, we have that $2k$ divides $r_{1,0}-r_{2,0}$. Let $i=\frac{r_{1,0}-r_{2,0}}{2k}$. Let 
\[
\rho_1'=\diag(\rho_1,\underbrace{\pi_{0}}_i)
\]
and 
\[
\rho_2'=\diag(\rho_1,\underbrace{\pi_{1}}_i).
\]
Since $\diag(\pi_0,0_{2k})\sim_u\diag(\pi_{\infty_1},\pi_{\infty_2},\pi_1)$, then
\[
\rho_1'\sim_u
\diag(\underbrace{\pi_{\infty_1}}_\ell, \underbrace{\diag(\pi_{\infty_1},\pi_{\infty_2})}_{r_{1,1}+i},0_{r_{2,0}},\pi_{x_1},\dots,\pi_{x_m},\underbrace{\pi_1}_i).
\]
Hence, by Case 1, $\rho_1'$ and $\rho'_2$ can be made unitarily equivalent by adding point representations. Since $\rho_1'$ (resp, $\rho'_2$) is obtained from $\rho_1$ (resp, $\rho_2$) by adding point representations, the thesis follows.
\end{proof}
\begin{remark}\label{remark:points}
The choice of how many points are needed depends only of $A$, $q$, and $\ell$. Since the range of the possible $K$-theories of maps $A\to M_q$ depends only on $q$ (since the values $r_{i,j}$ are bounded by $q$), there is a function $f\colon\NN\to\NN$ such that if $A$ is a generalised Razak block and $\rho_1,\rho_2\colon A\to M_q$ are representations with the same $K$-theory then there exist $x_1,\ldots,x_{f(q)}, y_1,\ldots,y_{f(q)}\in[0,1]$ such that 
 \[
 \diag(\rho_1,\pi_{x_1},\ldots,\pi_{x_{f(q)}})\sim_u \diag(\rho_2,\pi_{y_1},\ldots,\pi_{y_{f(q)}}).
 \] 
 \end{remark}
\subsubsection*{Traces}
A state $\tau$ on a \cstar-algebra $A$ such that $\tau(ab)=\tau(ba)$ for all $a,b\in A$ is a \emph{trace}. We denote the trace space of $A$ by $T(A)$. If $n\in\NN$, $\tau_n$ is the unique trace on $M_n$. If $A$ and $B$ are \cstar-algebras, $\sigma\in T(A)$ and $\tau\in T(B)$, we say that a $^*$-homomorphism sends $\sigma$ to $\tau$, and write 
\[
\varphi\colon (A,\sigma)\to (B,\tau),
\]
if for all $a\in A$ we have $\sigma(a)=\tau(\varphi(a))$. 

The trace space of (generalised) Razak blocks is not compact. Indeed, the traces $f\mapsto\tau_N(f(t))$ (where $N$ is either $nk$ or $2nk$ as appropriate) converge as $t\to1$ to a linear functional of norm $\frac{n-1}{n}<1$. However, $T(A)$ is contained in the $w^*$-closed convex hull of the extremal traces $\partial_eT(A)$, and these are in bijective correspondence via the GNS construction with the spectrum $\hat A$ of $A$. In fact, the `hull-kernel' topology on the space of irreducible representations coincides with the quotient topology supplied by the GNS map and the $w^*$-topology on $\partial_eT(A)$; so the correspondence is a homeomorphism. Therefore, every trace on a (generalised) Razak block corresponds to a unique Borel probability measure on $(0,1)\cup\{\infty_i\}$.

To be precise, fix a (generalised) Razak block $A$ and $\tau\in T(A)$. Define a measure $\mu_\tau$ by
\[
\mu_\tau(U)=\sup\{ \tau(f)\mid f\in (A)_+, \norm{f}\leq 1, \supp(f)\subseteq U\}
\]
for open sets $U \subseteq (0,1)\cup\{\infty_i\}$. Here, by $\supp(f)\subseteq U$ we mean that $\pi(f)=0$ for every $\pi\in \hat A \setminus U$, or, in other words, that $f(t)=0$ for $t\notin U$. In the case of $A_{n,k}$, this is the same as a Borel probability measure on $[0,1)$, or equally, a Borel probability measure $\mu$ on $[0,1]$ with $\mu(\{1\})=0$. In the case of $B_{n,k}$, $\mu_\tau$ is uniquely of the form
\[
\mu_\tau = \lambda_1\delta_{\infty_1} + \lambda_2\delta_{\infty_2} + \lambda_3\mu,
\]
where $\delta_t$ is the point mass at $t$, $\mu$ is a measure on $(0,1)$ and $\lambda_1+\lambda_2+\lambda_3=1$.

Conversely, if $A=A_{n,k}$ is a Razak block, to a Borel probability measure $\mu$ on $[0,1)$ we associate a trace $\tau_{\mu}\in T(A)$ by
\[
\tau_\mu(f)=\int_{[0,1)}\tau_{nk}(f(t))d\mu(t).
\]
If $A=B_{n,k}$ is a generalised Razak block, to a Borel probability measure $\mu$ on $(0,1)\cup\{\infty_i\}$ we associate the trace $\tau_\mu\in T(A)$ by
\[
\tau_\mu(f)=\int_{(0,1)}\tau_{2nk}(f(t))d\mu(t)+\tau_k(a_f)\mu(\infty_1)+\tau_k(b_f)\mu(\infty_2).
\] 
It is routine to check that $\tau_{\mu_\tau}=\tau$ and $\mu_{\tau_\mu}=\mu$.

A trace $\tau\in T(A)$ is called \emph{faithful} if for all $f\in A$ we have $\tau(ff^*)=0$ if and only if $f=0$. For (generalised) Razak blocks, $\tau$ is faithful if and only if $\mu_\tau(U)\neq 0$ whenever $U\subseteq(0,1)$ is a nonempty open set. If $A$ is a (generalised) Razak block, a trace $\tau\in T(A)$ is called \emph{diffuse} if it is associated to an atomless measure $\mu$ on $(0,1)$, that is, if $\tau=\tau_\mu$ and $\mu(\{x\})=0$ for all $x\in(0,1)\cup\{\infty_i\}$. We denote by $T_f$ and $T_{fd}$ the sets of all faithful, and faithful diffuse traces respectively.

\begin{remark}\label{remark:nontracial}
If $\varphi$ is a unital $^*$-homomorphism between unital \cstar-algebras $A$ and $B$, then for all $\tau\in T(B)$ there is $\sigma\in T(A)$ such that $\varphi\colon (A,\sigma)\to (B,\tau)$; that is, the pullback of a trace is always a trace. In the nonunital case, the pullback functional of a trace $\tau$ need not be a trace. For example, let $\varphi\colon A_{2,1}\to A_{2,2}$ be defined as
\[
\varphi(f)(t)=\begin{cases}\begin{pmatrix}
f(2t)&0\\
0&f(2t)
\end{pmatrix}& 0\leq t\leq 1/2\\
u(t)\begin{pmatrix}f(0)&0\\
0&0
\end{pmatrix}
u(t)^* & 1/2\leq t\leq 1
\end{cases}
\] 
where $u(1/2)$ is the permutation unitary that swaps the second and the third rows of matrices in $M_4$, $u(1)=1$ and $u(t)$ is any continuous path of unitaries connecting $u(1/2)$ to $1$. If $\tau\in T_f(A_{2,2})$ and $\sigma=\varphi^*(\tau)=\tau\circ\varphi$ is the pullback functional of $\tau$, then
\[
\norm{\sigma}=\mu_\tau([0,1/2])+\frac{1}{2}\mu_\tau([1/2,1])<1.
\]
If $A$ is a (generalised) Razak block and $\pi\colon A\to M_m$ is a representation, the pullback functional of the trace $\tau_m$ is a state (and therefore a trace) if and only if the number $r_{0}$ of Proposition~\ref{prop:repr} is $0$. We will use this in Proposition~\ref{prop:tpstarhoms}. 
\end{remark}

\subsection{The classes} 
We now introduce the classes we are going to work with.

Let
\[
\Obj_{R}=\{(A_{n,k},\tau)\mid n,k\in \NN, \tau\in T_{fd}(A_{n,k})\}
\]
and
\[
\Mor_{R}= \{\varphi\colon (A,\sigma)\to(B,\tau)\mid (A,\sigma),(B,\tau)\in\Obj_{R}\}.
\]
\begin{defin}
$\mathcal K_{\mathcal W}$ is the category with objects $\Obj_R$ and morphisms $\Mor_R$.
\end{defin}
Let $\mathcal P$ be the class of all prime numbers. A supernatural number of infinite type $\bar p$ is an expression of the form $\bar p=\prod_{p\in\mathcal P} p^{\ell_p}$, where $\ell_p\in\{0,\infty\}$. We say that an integer $k\in\zet$ divides $\bar p$ if every prime in the unique factorisation of $|k|$ corresponds to a prime whose $\ell_p$ is infinite. $0$ does not divide any supernatural number, while $-1$ and $1$ divide all of them.
Let
\[
\Obj_{GR}=\{(B_{n,k}, \tau)\mid \tau\in T_{fd}(B_{n,k})\}.
\]
Let 
\[
\Mor_{GR,0}= \{\varphi\colon (A,\sigma)\to(B,\tau)\mid (A,\sigma),(B,\tau)\in\Obj_{GR}\},
\]
and
\[
\Mor_{GR,1}= \{\varphi\colon (A,\sigma)\to(B,\tau)\mid (A,\sigma),(B,\tau)\in\Obj_{GR}\text{ and } |K_0(\varphi)|=1\}.
\]
For a supernatural number of infinite type $\bar p$, let
\[
\Mor_{GR,\bar p}= \{\varphi\colon (A,\sigma)\to(B,\tau)\mid (A,\sigma),(B,\tau)\in\Obj_{GR}\text{, }K_0(\varphi)\text{ divides }\bar p\}.
\]
\begin{defin}
$\mathcal K_{0}$ is the category with objects $\Obj_{GR}$ and morphisms $\Mor_{GR,0}$. $\mathcal K_{1}$ is the category with objects $\Obj_{GR}$ and morphisms $\Mor_{GR,1}$. If $\bar p$ is a supernatural number of infinite type, $\mathcal K_{\bar p}$ is the category with objects $\Obj_{GR}$ and morphisms $\Mor_{GR,\bar p}$.
\end{defin}

\begin{proposition}\label{prop:WPPCCP}
The classes $\mathcal K_\mathcal W$, $\mathcal K_0$, $\mathcal K_1$ and $\mathcal K_{\bar p}$, where $\bar p$ is a supernatural number of infinite type, have the WPP and the CCP.
\end{proposition}
\sloppy
\begin{proof}
CCP is obvious, as all functions and predicates involved are $1$-Lipschitz on the unit ball. For WPP, notice that there are only countably many (generalised) Razak blocks and each of them is separable. The transition maps of Proposition~\ref{prop:onetoanother} give the WPP.
\end{proof}
\fussy

\section{Diagonal maps} \label{S.diagonal}
The following definitions are designed to identify a class of maps between (generalised) Razak block which are treatable. If $A$ and $B$ are (generalised) Razak blocks, $t\in[0,1]$ and $\varphi\colon A\to B$ is a $^*$-homomorphism, we let $\pi_{\varphi,t}$ be the representation obtained by 
\[
f\mapsto\varphi(f)(t).
\]
We define the representations $\pi_{\varphi,\infty}$ (if $B=A_{n,k}$ is a Razak block) and $\pi_{\varphi,\infty_1}$ and $\pi_{\varphi,\infty_2}$ (if $B=B_{n,k}$ is a generalised Razak block): $\pi_{\varphi,\infty}$ is the $k$-dimensional representation such that 
\[
\pi_{\varphi,0}=\diag(\underbrace{\pi_{\varphi,\infty}}_n)
\]
and $\pi_{\varphi,\infty_1}, \pi_{\varphi,\infty_2}$ are the $k$-dimensional representations such that 
\[
\pi_{\varphi,0}=\diag(\underbrace{\pi_{\varphi,\infty_1},\pi_{\varphi,\infty_2}}_n).
\] 
\begin{defin}
Let $A\subseteq C([0,1],M_n)$ and $B\subseteq C([0,1],M_m)$ be \cstar-algebras, and let $\varphi\colon A\to B$ be a $^*$-homomorphism. A point $t\in [0,1]$ is said to be \emph{regular} for $\varphi$ if there are $s_1,\ldots,s_j\in [0,1]$  such that 
\[
\pi_{\varphi,t}\sim_u\diag(\pi_{s_1},\ldots,\pi_{s_j}).
\]
\end{defin}

In the class of $m$-dimensional representations of $A$, the ones unitarily equivalent to those of the form $\diag(\pi_{s_1},\ldots,\pi_{s_j})$ for some points $s_1,\ldots,s_j\in [0,1]$ form a closed set in the hull-kernel topology. The following is then immediate.
\begin{proposition} \label{prop:regclosed}
Let $A\subseteq C([0,1],M_n)$ and $B\subseteq C([0,1],M_m)$ be \cstar-algebras and let $\varphi\colon A\to B$ be a $^*$-homomorphism. The set of regular points is closed.\qed
\end{proposition}

\begin{defin}\label{defin:diag}
Let $A\subseteq C([0,1],M_n)$ and $B\subseteq C([0,1],M_m)$ be \cstar-algebras. A $^*$-homomorphism $\varphi\colon A\to B$ is called \emph{diagonal} if $n$ divides $m$ and there are continuous maps $\xi_i\colon [0,1]\to[0,1]$, for $i\leq m/n$, such that  $\xi_i\leq\xi_{i+1}$ for all $i$, and
\[
\pi_{\varphi,t}\sim_u\diag(\pi_{\xi_1(t)},\ldots,\pi_{\xi_{m/n}(t)}), \text{ for all } t\in [0,1].
\]
The maps $\{\xi_i\}$ are said to be \emph{associated} to $\varphi$. 
\end{defin}

As any two faithful diffuse probability measures on $(0,1)$ can be sent one to another via a homeomorphism of $[0,1]$, the same can be said for traces on (generalised) Razak blocks. 

\begin{proposition}\label{prop:onetoanother}
Let $A$ be a (generalised) Razak block and let $\sigma,\tau\in T_{fd}(A)$. Then there is a diagonal automorphism $\varphi\colon (A,\sigma)\to (A,\tau)$ which is trivial on $K$-theory.
\end{proposition}

\begin{proof}
For every $t\in [0,1]$, there exists (by faithfulness) a unique $s_t\in[0,1]$ such that $\mu_\tau([0,t])=\mu_\sigma([0,s_t])$. The function $\xi=\xi_{\sigma\mapsto\tau}\colon[0,1]\to[0,1]$ defined by $t\mapsto s_t$ is a homeomorphism, and the map $\varphi=\xi^*$, that is, $\varphi(f)(t)=f(\xi(t))$, is the required automorphism of $A$. That $K_0(\varphi)=1$ follows from Remark~\ref{remark:generator}.
\end{proof}

We call such an automorphism a \emph{transition map} and denote it by $\varphi_{\sigma\mapsto\tau}$.
As we have seen in Remark~\ref{remark:nontracial}, not all $^*$-homomorphisms are trace preserving. Yet, this is (often) the case for diagonal maps:
\sloppy
\begin{proposition}\label{prop:tpstarhoms}
Let $A$ and $B$ be (generalised) Razak blocks, and let $\varphi\colon A\to B$ be a $^*$-homomorphism. The following are equivalent:
\begin{enumerate}
\item\label{p1} $\varphi$ is diagonal with associated maps $\{\xi_i\}_{i\leq j}$, and
\[
\mu_\lambda(\{t\mid\exists i (\xi_i(t)=s)\})=0 \quad\text{ for every }\: s\in[0,1],
\]
$\mu_\lambda$ being the Lebesgue measure;
\item\label{p2} there are $\sigma\in T_{fd}(A)$ and $\tau\in T_{fd}(B)$ such that 
\[
\varphi\colon (A,\sigma)\to(B,\tau);
\]
\item\label{p3} for all $\tau\in T_{fd}(B)$ there is $\sigma\in T_{fd}(A)$ such that 
\[
\varphi\colon (A,\sigma)\to(B,\tau).
\]
\end{enumerate}
\end{proposition}
\fussy

\begin{proof}
We show that \eqref{p1}$\Rightarrow$\eqref{p3} and that \eqref{p2}$\Rightarrow$\eqref{p1}, as \eqref{p3}$\Rightarrow$\eqref{p2} is obvious. We only give the proof in the case where $A=A_{n,k}$ and $B=A_{n',k'}$ are Razak blocks, and leave to the reader to check that the proof generalises.

\eqref{p1}$\Rightarrow$\eqref{p3}: Since $\varphi$ is diagonal, $f(t)$ is $nk$-dimensional, therefore $nkj=n'k'$. Fix $\tau\in T_{fd}(B)$. Let $\sigma$ be the pullback functional of $\tau$, and $\mu_\sigma$ be the Borel measure on $[0,1)$ associated to $\sigma$. The goal is to show that $\mu_\sigma$ is a faithful diffuse probability measure. Since $\varphi$ is injective by Proposition~\ref{prop:Ktheory}, $\sigma$ is faithful, and so is $\mu_\sigma$. 
For any open (hence any Borel) set $U\subseteq[0,1)$, we have
\[
\mu_\sigma(U)=\frac{1}{j}\sum_{i\leq j}\mu_\tau(\{t\mid\xi_i(t)\in U\}).
\]
In particular, $\mu_\sigma([0,1))=1$, so $\mu_\sigma$ is indeed a probability measure. Notice that this also shows that if $\mu_\tau(\{t\mid\exists i (\xi_i(t)=s)\})=0$ then $\mu_\sigma(\{s\})=0$. Since $\mu_\lambda(\{t\mid\exists i (\xi_i(t)=s)\})=0$ for every $s\in[0,1]$, and $\mu_\lambda$ and $\mu_\tau$ are uniformly continuous with respect to each other (see Proposition~\ref{prop:measuring}), then $\mu_\tau(\{t\mid\exists i (\xi_i(t)=s)\})=0$ for every $s\in[0,1]$. Hence $\mu_\sigma(\{s\})=0$ for all $s\in [0,1)$, and therefore $\sigma$ is diffuse.

\eqref{p2}$\Rightarrow$\eqref{p1}. Let $\tau\in T_{fd}(B)$ and $\sigma\in T_{fd}(A)$ such that $\varphi\colon (A,\sigma)\to (B,\tau)$. Let $X$ be the set of regular points of $\varphi$.
\begin{claim}
$X$ is dense.
\end{claim}
\begin{proof}
Suppose not and let $U\subseteq[0,1]\setminus X$ be a nonempty open set. Since $\mu_\tau$ is faithful, there is $\varepsilon>0$ such that $\mu_\tau(U)>\varepsilon$. Consider $\mathcal I=\{f\in A\mid a_f=0\}$. Notice that since $\sigma$ is a faithful diffuse trace then $\sup_{\norm{f}\leq 1, f\in\mathcal I}\sigma(f)=1$. On the other hand, if $t\in U$, then 
\[
\pi_{\varphi,t}\sim_u\diag(\pi_{s_1},\ldots,\pi_{s_j},\underbrace{\pi_\infty}_{r_1},0_{r_0})
\]
for some $s_1,\ldots,s_j\in [0,1)$, $j_i\in\NN$, and therefore for all $f\in\mathcal I$ with $\norm{f}\leq 1$ we have 
\[
\tau(\varphi(ff^*))\leq\mu_\tau([0,1]\setminus U)+\frac{n'k'-1}{n'k'}\mu_\tau(U).
\]
In particular, there is no contraction $f\in\mathcal I$ such that $\tau(\varphi(ff^*))\geq (1-\varepsilon)+\frac{n'k'-1}{n'k'}\varepsilon$, which contradicts the fact that $\tau\circ\varphi=\sigma$.
\end{proof}

By Proposition~\ref{prop:regclosed}, $X$ is closed, hence $X=[0,1]$. As there exists one regular point, $nk$ divides $n'k'$. Let $j=\frac{n'k'}{nk}$. For every $t\in [0,1]$, let $s_1^t,\ldots,s_j^t\in[0,1]$ be such that 
$
\pi_{\varphi,t}\sim_u\diag(\pi_{s_1^t},\ldots,\pi_{s_j^t}).
$
 Define $\xi_1,\ldots,\xi_j\colon [0,1]\to[0,1]$ inductively by $\xi_1(t)=\min \{s_i^t\}$ and $\xi_i(t)=\min \{\{s_i^t\}\setminus\{\xi_1(t),\ldots\xi_{i-1}(t)\}\}$ for $i>1$, where the sets $\{s_1^t,\ldots,s_j^t\}$ are considered as multisets (i.e., if an element appears twice, we count it twice).

The maps $\xi_i$ are continuous and witness that $\varphi$ is diagonal, since $\xi_i\leq\xi_{i+1}$ for every $i$. Finally, suppose that there is $s\in[0,1]$ such that $\mu_\lambda(\{t\mid\exists i (\xi_i(t)=s)\})>0$. Once again using the uniform continuity of $\mu_\lambda$ and $\mu_\tau$ with respect to each other, it follows that $\mu_\tau(\{t\mid\exists i (\xi_i(t)=s)\})>0$, and therefore $\mu(\sigma(\{s\}))>0$, a contradiction to the diffuseness of $\sigma$. This concludes the proof.
\end{proof}
\subsection{Diameter}\label{SS.Diameter}
The next notion measures the `amplitude' of a diagonal map.
\begin{defin}\label{defin:diam}
The \emph{diameter} of a map $\xi\colon[0,1]\to[0,1]$ is the number
\[
\partial(\xi) = \sup_{s,t\in [0,1]}|\xi(s)-\xi(t)|.
\]
Let $A\subseteq C([0,1],M_n)$ and $B\subseteq C([0,1],M_m)$ and let $\varphi\colon A\to B$ be a diagonal $^*$-homomorphism with associated maps $\{\xi_i\}$. The \emph{diameter} of $\varphi$ is the number
\[
\partial(\varphi) = \sup_i\partial(\xi_i).
\]
\end{defin}
We record some basic results.

\begin{lemma} \label{lemma:diameterfacts}
\sloppy
\begin{enumerate}[(i)]
\item \label{diam1} Let $\{\xi_i\}_{i\leq j}$ be continuous maps $[0,1]\to[0,1]$ with $\sup_i\partial(\xi_i)<\varepsilon$. Define maps $\xi_i'$, for $i\leq j$, by
\begin{align*}
\xi'_1(t)&=\min\{\xi_1(t),\ldots,\xi_j(t)\},\\
\xi_{i+1}'(t)&=\min (\{\xi_1(t),\ldots,\xi_j(t)\}\setminus \{\xi'_1(t),\ldots,\xi'_i(t)\}),
\end{align*}
viewing the sets $\{\xi_i(t)\}_{i\leq j}$, for $t\in[0,1]$, as multisets (i.e., if an element appears twice, we count it twice). Then $\sup_i\partial(\xi_i')<2\varepsilon$.
\item \label{diam2} If $\varphi\colon A\to B$ and $\psi\colon B\to C$ are diagonal $^*$-homomorphisms between (generalised) Razak blocks, then \[\partial(\psi\circ\varphi)\le\partial(\varphi).\]
\item \label{diam3} For any diagonal map $\varphi\colon A\to B$ between (generalised) Razak blocks and any $\varepsilon>0$, there is $\delta>0$ such that, for any diagonal map $\psi\colon B\to C$ from $B$ to a (generalised) Razak block $C$, if $\partial(\psi)<\delta$, then $\partial(\psi\circ\varphi)<\varepsilon$.
\end{enumerate}
\fussy
\end{lemma}

\begin{proof}
\eqref{diam1} is an easy calculation. For \eqref{diam2}: If $\xi_1,\ldots,\xi_{j}$ and $\zeta_1,\ldots,\zeta_{j'}$ are the maps associated to $\varphi$ and $\psi$ respectively, then $\psi\circ\varphi$ has associated maps $\xi_i\circ\zeta_k$. \eqref{diam3} is true since the maps $\xi_i$ associated to $\varphi$ are uniformly continuous. 
\end{proof}

We now construct diagonal maps between (generalised) Razak blocks with small diameters. 
\subsubsection{Razak blocks}\label{SS.Maprazak}
The following is adapted from \cite[Proposition 3.1]{Jacelon:W}, where it was stated and proved in the case $p=2$.

\begin{proposition}\label{prop:bhishansmaps}
Let $n,k,p,k'\in\NN$ with $p>0$. Then
\begin{enumerate}
\item\label{emb:c0} there is a diagonal $\varphi_{n,k,p}\colon A_{n,k}\to A_{pn,(pn-1)k}$ with $\partial(\varphi_{n,k,p})\leq\frac{1}{p}$;
\item\label{emb:c1} there is a diagonal $\psi_{n,k,k'}\colon A_{n,k}\to A_{n,kk'}$. 
\end{enumerate}
Moreover, these maps satisfy the equivalent conditions of Proposition~\ref{prop:tpstarhoms}.
\end{proposition}

\begin{proof}
\eqref{emb:c0}: Let $b=pn-1$, and let $a_\varphi$ be the $k(n-1)+nk(p-1)=bk$-dimensional representation given by
\[
a_{\varphi}=\diag(\underbrace{\pi_\infty}_{n-1},\pi_{1/p},\pi_{2/p},\ldots,\pi_{(p-1)/p}).
\]
Let $\xi_i$, for $1\leq i\leq pb$, be continuous finite-to-one maps such that 
\begin{equation*}
 \xi_i(0)=
 \begin{cases}
 0, & \text{if}\ 1\leq i\leq b-p+1 \\
 1/p, & \text{if }\ b-p+1<i\leq 2b-p+2\\
 \vdots&\vdots\\
 (p-1)/p& \text{if}\ (p-1)b-p+(p-1)<i\leq pb,
 \end{cases}
\end{equation*}
and 
\[ \xi_i(1)=(j+1)/p\text{ if }jb< i \leq(j+1)b\text{, for }0\leq j\leq p-1.
\]
Additionally, we require that $\partial(\xi_i)\leq\frac{1}{p}$ for all $i$. (If $\xi_i(0)<\xi_i(1)$, just take $\xi_i$ to be linear. If $\xi_i(0)=\xi_i(1)$, just pick a piecewise linear finite-to-one function of small diameter.)

 Let $\psi\colon A_{n,k}\to C([0,1],M_{pn(pn-1)k})$ be given by \[\pi_{\psi,t}=\diag(\pi_{\xi_1(t)},\ldots,\pi_{\xi_{pb}(t)}) \text{ for all } t\in [0,1].\]
Noticing that 
\[
\pi_{\psi,0}=\diag(\underbrace{\pi_0}_{b-p+1},\underbrace{\pi_{1/p}}_{b+1},\ldots,\underbrace{\pi_{(p-1)/p}}_{b+1})
\]
and that 
\[
\pi_{\psi,1}=\diag(\underbrace{\pi_{1/p}}_{b},\underbrace{\pi_{2/p}}_b,\ldots,\underbrace{\pi_1}_{b}),
\]
we have that 
\[
\pi_{\psi,0}\sim_u\diag(\underbrace{a_{\varphi}}_{pn}) \text{ and } \pi_{\psi,1}\sim_u\diag(\underbrace{a_{\varphi}}_{b}, 0_{bk}).
\]
Let $u\in C([0,1],M_{pn(pn-1)k})$ be a unitary which at the boundary points $0$ and $1$ coincides with the two unitaries witnessing the $\sim_u$ relations above. Then
\[
\varphi_{n,k,p}=\Ad(u)\circ\psi\colon A_{n,k}\to A_{pn,(pn-1)k}
\] is as required. 

\eqref{emb:c1}: Consider the amplification map \[\iota_{k'}\colon C([0,1],M_{nk})\to C([0,1],M_{nkk'})=C([0,1],M_{k'}\otimes M_{nk})\] given by $a\mapsto 1_{k'}\otimes a$. Let $a_\psi=\diag(\underbrace{\pi_\infty}_{k'})$. Notice that 
\[
\pi_{\iota_{k'},1}\sim_u\diag(\underbrace{a_\psi}_{n-1},0_{kk'}),
\]
 hence there is a unitary $u\in C([0,1],M_{nkk'})$ such that $\psi_{n,k,k'}=\Ad(u)\circ\iota_{k'}$ is the required map.

Finally, since all the maps considered are finite-to-one, the equivalent conditions of Proposition~\ref{prop:tpstarhoms} are satisfied.
\end{proof}
The following definition was given in \cite{Jacelon:W}.
\begin{defin}\label{defin:W}
Let $n_1=1=k_1$. For $i>1$, let $n_i=(i-1)n_{i-1}$ and $k_i=(n_{i}-1)k_{i-1}$. Let $A_i=A_{n_i,k_i}$ and let $\varphi_i=\varphi_{n_i,k_i,i}\colon A_i\to A_{i+1}$ be the map defined in Proposition~\ref{prop:bhishansmaps}\eqref{emb:c0}. Define $\mathcal W=\lim(A_i,\varphi_i)$.
\end{defin}
The algebra $\mathcal W$ is automatically simple and monotracial (see \cite[Proposition 3.5]{Jacelon:W}).
\begin{remark}
Classification methods (\cite[Theorem 1.1]{Razak}) give that every inductive limit of Razak blocks which has a unique tracial state and is simple must be isomorphic to $\mathcal W$, showing that the latter is `generic' in some sense among inductive limits of Razak blocks. By showing that $\mathcal W$ is the \Fraisse limit of $\mathcal K_\mathcal W$ we obtain the same result, more formally in a model theoretic sense, without making use of classification.
\end{remark}
\subsubsection{Generalised Razak blocks}
For generalised Razak blocks we in addition ask our maps to respect precise $K$-theoretical constraints.
By Proposition~\ref{prop:Ktheory} and Remark~\ref{remark:generator}, for any $n,k\in\mathbb{N}$ we have that $K_0(B_{m,l})\cong\mathbb Z$, and if $B_{n,k}$ and $B_{n',k'}$ are generalised Razak blocks and $\varphi\colon B_{n,k}\to B_{n',k'}$ is a $^*$-homomorphism, we identify $K_0(\varphi)$ with the integer $[\varphi(p_{n,k})]-[1_{\tilde B_{n',k'}}]$. We compute this integer in terms of the representation theory of $\varphi$.
Let $r_{i,\ell}=r_{\pi_{\varphi,\infty_i},\ell}$ and $j_{i}=j_{\pi_{\varphi,\infty_i}}$ for $i=1,2$, $\ell=0,1,2$, be the values provided for the representations $\pi_{\varphi,\infty_1}$ and $\pi_{\varphi,\infty_2}$ by Proposition~\ref{prop:repr}.

\begin{proposition}\label{prop:Ktheory2}
If $\varphi\colon B_{n,k}\to B_{n',k'}$ is a $^*$-homomorphism, then $K_0(\varphi)$ is completely determined by the values $r_{i,j}$, for $i,j=1,2$. In particular
\[
K_0(\varphi)=\frac{1}{2}(r_{1,1}-r_{1,2}+r_{2,2}-r_{2,1}).
\]
If $\varphi$ is diagonal, then $r_{1,1}+r_{2,1}=r_{1,2}+r_{2,2}$, hence $K_0(\varphi)=r_{1,1}-r_{2,1}$.
\end{proposition}

\begin{proof}
Recall from Remark~\ref{remark:generator} that
\[
K_0(B_{n,k}) = \Span_\zet\{(k+1,k-1,1)-(k,k,1)\} = \Span_\zet\{[p_{n,k}]-[1_{\tilde B_{n,k}}]\}.
\]
Let $N\in \zet$ such that $K_0(\varphi)=N$. We will compute $N$ in terms of the values $r_{i,j}$. Extending $\varphi$ to a unital $^*$-homomorphism $M_2(\tilde B_{n,k}) \to M_2(\tilde B_{n',k'})$ we have
\begin{align*}
[\varphi(p_{n,k})]-[1_{\tilde B_{n',k'}}] &= N([p_{n,k}]-[1_{\tilde B_{n',k'}}])\\ &= N(k'+1,k'-1,1) - N(k',k',1)\\ &= (N,-N,0),
\end{align*}
so $[\varphi(p_{n,k})] = (N,-N,0) + (k',k',1) = (N+k',-N+k',1)$. Since $N=\frac{1}{2}((N+k')-(-N+k'))$, it follows that
\begin{align*}
N &= \frac{1}{2}(\rank(\pi_{\infty_1}(\varphi(p_{n,k}))) - \rank(\pi_{\infty_2}(\varphi(p_{n,k}))))\\
&= \frac{1}{2}(\rank(\pi_{\varphi,\infty_1}(p_{n,k})) - \rank(\pi_{\varphi,\infty_2}(p_{n,k})))\\
& = \frac{1}{2}(2nkj_1 + (k+1)r_{1,1} + (k-1)r_{1,2} + r_{1,0}\\ &- 2nkj_2 - (k+1)r_{2,1} - (k-1)r_{2,2} - r_{2,0})\\
&= \frac{1}{2}(r_{1,1}-r_{1,2}+r_{2,2}-r_{2,1})
\end{align*}
since 
\[
2nkj_1 + kr_{1,1} + kr_{1,2} + r_{1,0}=k'=2nkj_2 + kr_{2,1} + kr_{2,2} + r_{2,0}.
\]
If $\varphi$ is diagonal, let $\xi_i$ be the maps associated to $\varphi$. By Proposition~\ref{prop:repr}, for $i=1,2$,  there are $s^i_{1},\ldots,s^i_{j_i}\in[0,1)$ such that
\[
\pi_{\varphi,\infty_i}\sim_u\diag(\pi_{s^i_1},\ldots,\pi_{s^i_{j_i}},\underbrace{\pi_{\infty_1}}_{r_{i,1}},\underbrace{\pi_{\infty_2}}_{r_{i,2}},0_{r_{i,0}})
\]
Let $m_0=|\{m\mid \xi_m(0)=0\}|$, $m_1=|\{m\mid \xi_m(0)=1\}|$ and $p_i=|\{p\mid s^i_p=0\}|$. Then 
\[
nm_0+(n-1)m_1=nn'(p_1+p_2)+n'(r_{1,1}+r_{1,2})=nn'(p_1+p_2)+n'(r_{2,1}+r_{2,2}),
\]
hence we have $r_{1,1}+r_{1,2}=r_{2,1}+r_{2,2}$, and therefore the thesis.
\end{proof}
\begin{corollary}\label{cor:Ktheoryreverting}
Let $A$ and $B$ be generalised Razak blocks. Suppose that there is a $^*$-homomorphism $\varphi\colon A\to B$ with $K_0(\varphi)=j$. Then there is a $^*$-homomorphism $\tilde\varphi\colon A\to B$ with $K_0(\tilde\varphi)=-j$. If $\varphi$ is diagonal, so is $\tilde\varphi$, and the two have the same associated maps.
\end{corollary}
\begin{proof}
Let $n,k\in\NN$ be such that $B=B_{n,k}$. Let $u\in M_{2nk}$ be a unitary such that 
\[
u\begin{pmatrix}
a&0\\
0&b
\end{pmatrix}
 u^*=\begin{pmatrix}
b&0\\
0&a
\end{pmatrix}
\]
for all $a,b\in M_{nk}$. $\Ad(u)\circ\varphi$ is the required $^*$-homomorphism.
\end{proof}

The generalised version of Proposition~\ref{prop:bhishansmaps} takes $K$-theory into account.
\begin{proposition}\label{prop:genmaps}
Let $n,k,p,k'\geq 2$.
\begin{enumerate}
\item\label{genmaps1} If $p$ is odd, then for every $0\leq j\leq (n-1)$ there is a diagonal $\varphi_{n,k,p,j}\colon B_{n,k}\to B_{pn,(pn-1)k}$ with $\partial(\varphi_{n,k,p,j})\leq\frac{1}{p}$ and
\[
K_0(\varphi_{n,k,p,j})=2j-(n-1).
\] 
\item\label{genmaps2} For all $0\leq j\leq k'$ there is a diagonal $\psi_{n,k,k',j}\colon B_{n,k}\to B_{n,kk'}$ such that $K_0(\psi_{n,k,k',j})=2j-k'$. 
\item\label{genmaps3} For every $j$ with $|j|\leq (n-1)k'$ there is a diagonal $\rho_{n,k,j}\colon B_{n,k}\to B_{nk,(nk-1)k'}$ such that $K_0(\rho_{n,k,j})=j$.
\end{enumerate}
Moreover, these maps satisfy the equivalent conditions of Proposition~\ref{prop:tpstarhoms}.
\end{proposition}
\begin{proof}
\eqref{genmaps1}: let $b=pn-1$ and $\xi_i$, for $i\leq bp$, be the continuous functions $\xi_i\colon [0,1]\to[0,1]$ as defined in Proposition~\ref{prop:bhishansmaps}\eqref{emb:c0}. Define $\varphi\colon B_{n,k}\to C([0,1],M_{2pnbk})$ by 
\[
\pi_{\varphi,t}=\diag(\pi_{\xi_1(t)},\ldots,\pi_{\xi_{pb}(t)})
\]
for all $t\in [0,1]$. Notice that 
\[
\pi_{\varphi,0}\sim_u\diag(\underbrace{\pi_{\infty_1}}_{np(n-1)},\underbrace{\pi_{\infty_2}}_{np(n-1)}, \underbrace{\pi_{1/p}}_{pn}, \ldots, \underbrace{\pi_{(p-1)/p}}_{pn})
\]
and 
\[
\pi_{\varphi,1}\sim_u\diag(\underbrace{\pi_{\infty_1}}_{b(n-1)},\underbrace{\pi_{\infty_2}}_{b(n-1)}, \underbrace{\pi_{1/p}}_{b}, \ldots, \underbrace{\pi_{p-1)/p}}_{b}, 0_{b2k}).
\]
 If $0\leq j\leq n-1$ let 
\[
a_{\varphi,j}=\diag(\underbrace{\pi_{\infty_1}}_{j},\underbrace{\pi_{\infty_2}}_{n-1-j}, \pi_{1/p}, \pi_{3/p}, \ldots, \pi_{(p-2)/p})
\]
and 
\[
b_{\varphi,j}=\diag(\underbrace{\pi_{\infty_1}}_{n-1-j},\underbrace{\pi_{\infty_2}}_{j}, \pi_{2/p}, \pi_{4/p}, \ldots, \pi_{(p-1)/p}).
\]
Notice that $\pi_{\varphi,0}\sim_u\diag(\underbrace{a_{\varphi,j}}_{pn}, \underbrace{b_{\varphi,j}}_{pn})$ and $\pi_{\varphi,1}\sim_u\diag(\underbrace{a_{\varphi,j}}_b,0_{bk},\underbrace{b_{\varphi,j}}_b,0_{bk})$, therefore we can find a unitary $u_j$ such that the map $\varphi_{n,k,p,j}=\Ad(u_{j})\circ\varphi$ is as required, as $K_0(\varphi_{n,k,p,j})=2j-(n-1)$.

\eqref{genmaps2}: Let $\iota_{k'}$ be the amplification map as in Proposition~\ref{prop:bhishansmaps}\eqref{emb:c1}.  Let 
\[
a_{\psi,j}=\diag(\underbrace{\pi_{\infty_1}}_{j},\underbrace{\pi_{\infty_2}}_{k'-j})\text{ and }
b_{\psi,j}=\diag(\underbrace{\pi_{\infty_1}}_{k'-j},\underbrace{\pi_{\infty_2}}_{j}).
\]
Then 
\[
\pi_{\iota_{k'},0}\sim_u\diag(\underbrace{a_{\psi,j}}_n,\underbrace{b_{\psi,j}}_n)\text{ and }\pi_{\iota_{k'},1}\sim_u\diag(\underbrace{a_{\psi,j}}_{n-1},0_{kk'},\underbrace{b_{\psi,j}}_{n-1},0_{kk'}).
\]
Therefore there is a unitary $u_j$ such that the map $\psi_{n,k,k',j}=\Ad(u_j)\circ\psi\colon B_{n,k}\to B_{n,kk'}$ has  $K$-theory equal to $2j-k'$.

\eqref{genmaps3}: Let $\xi_1,\ldots,\xi_{k'(nk-1)}\colon [0,1]\to[0,1]$ be continuous maps such that
\[
\xi_i(0)=\begin{cases}
0&\text{ if }i\leq k'(n-1)\\
1&\text{else}
\end{cases},\,\,\,\,\,\,\,\xi_i(1)=1
\]
and let $\rho\colon B_{n,k}\to M_{2nk(nk-1)k'}$ be given by $\rho(f)=\diag(f\circ\xi_1,\ldots,f\circ\xi_{(nk-1)k'})$. Fix $j$ with $|j|\leq (n-1)k'$. In each of the three cases below, we will define $a_{\rho,j}$ and $b_{\rho,j}$ such that
\[
\pi_{\rho,0}\sim_u\diag(\underbrace{a_{\rho,j}}_{nk},\underbrace{b_{\rho,j}}_{nk})
\]
and
\[
\pi_{\rho,1}\sim_u\diag(\underbrace{a_{\rho,j}}_{nk-1},0_{(nk-1)k'},\underbrace{b_{\rho,j}}_{nk-1},0_{(nk-1)k'}),
\]
giving a map $\rho_j=\Ad(u_j)\circ\rho\colon B_{n,k}\to B_{nk,(nk-1)k'}$ for a suitable unitary $u_j\in C([0,1],M_{2nk(nk-1)k'})$. We will be done once we have computed the $K$-theory using Proposition~\ref{prop:Ktheory2}.

\noindent \underline{\textbf{Case 1:}} $(n-1)k'-j=2r$. Let 
\[
a_{\rho,j}= \diag(\underbrace{\pi_{\infty_1}}_{(n-1)k'-r},\underbrace{\pi_{\infty_2}}_{r},0_{(k-1)k'})
\] and 
\[
b_{\rho,j}= \diag(\underbrace{\pi_{\infty_2}}_{(n-1)k'-r},\underbrace{\pi_{\infty_1}}_{r},0_{(k-1)k'}).
\]
Then
\[
K_0(\rho_j)=\frac{1}{2}((n-1)k'-r-r+(n-1)k'-r-r)=(n-1)k'-2r=j.
\]

\noindent \underline{\textbf{Case 2:}} $(n-1)k'-j=1$. Let 
\[
a_{\rho,j}= \diag(\underbrace{\pi_{\infty_1}}_{(n-1)k'},\pi_{\infty_2},0_{(k-1)k'-k})
\] 
and 
\[
b_{\rho,j}= \diag(\underbrace{\pi_{\infty_2}}_{(n-1)k'-1},0_{(k-1)k'+k})
\]
Then $K_0(\rho_j)=\frac{1}{2}((n-1)k'-1+(n-1)k'-1)=(n-1)k'-1=j$.

\noindent \underline{\textbf{Case 3:}} $(n-1)k'-j=2r+1$, $r>0$. Let 
\[
a_{\rho,j}= \diag(\underbrace{\pi_{\infty_1}}_{(n-1)k'-(r+1)},\underbrace{\pi_{\infty_2}}_{r},0_{(k-1)k'+k})
\] and 
\[
b_{\rho,j}= \diag(\underbrace{\pi_{\infty_2}}_{(n-1)k'-r},\underbrace{\pi_{\infty_1}}_{r+1},0_{(k-1)k'-k}).
\] 
Then $K_0(\rho_j)=\frac{1}{2}((n-1)k'-(r+1)-r+(n-1)k'-r-(r+1))=(n-1)k'-(2r+1)=j$.

Lastly, since all the maps involved are finite-to-one, all the constructed maps satisfy the equivalent conditions of Proposition~\ref{prop:tpstarhoms}.
\end{proof}

We intend to define a `fast-enough' sequence of generalised Razak blocks. Let $A_0=B_{2,1}$. If $A_i=B_{n_i,k_i}$ has been defined, let $p_i$ be an odd number with the property that for all $^*$-homomorphisms with the same $K$-theory $\rho_1,\rho_2\colon A_{j}\to A_{i+1}$, for $j\leq i$, then it is enough to add $\frac{p_i-1}{2}$ point representations to make $\rho_1$ and $\rho_2$ unitarily equivalent. This is possible by Lemma~\ref{lemma:KtheoryGen2} and Remark~\ref{remark:points}.
\begin{defin}\label{defin:Z0}
Let $n_1=2$ and $k_1=1$. If $i>1$, define $n_i=p_{i-1}n_{i-1}$ and $k_i=(n_i-1)k_{i-1}$, where $p_i$ is defined as in the above paragraph. Let $A_i=B_{n_i,k_i}$ and $\varphi_i=\varphi_{n_i,k_i,p_i,n_i/2}$ be the map defined in Proposition~\ref{prop:genmaps}\eqref{genmaps1}. Define $\mathcal Z_0=\lim(A_i,\varphi_i)$.
\end{defin}
\begin{remark}\label{remark:Z0UHF}
Again by classification (\cite{GongLinNiu.Class1,GongLin.Class2} or \cite{Robert:2010qy}), we have that if $A$ is a limit of generalised Razak blocks which is simple, monotracial, and whose $K_0$ is $\mathbb Z$, then $A$ is isomorphic to $\mathcal Z_0$. Similarly, if $p$ is a prime number and $p^\infty$ the supernatural number of infinite type whose only factor is $p$, the algebra $\mathcal Z_0\otimes M_{p^\infty}$ can be obtained by combining maps of the form $\varphi_{n_i,k_i,p_i,n_i/2}$ from Proposition~\ref{prop:genmaps}\eqref{genmaps1} (to obtain simplicity and unique trace), and maps of the form $\psi_{n,k,p,p}$ from Proposition~\ref{prop:genmaps}\eqref{genmaps2} (to obtain a limit whose $K_0$ is $\mathbb Z[1/p]$). We will show that such objects are generic by showing they are the \Fraisse limits of their respective classes.

If we consider a direct system of generalised Razak blocks $(A_i,\varphi_i)$ whose limit is simple, monotracial, and such that for infinitely many $i$, $\varphi_i$ has $K$-theory equal to $0$, then $\lim_i(A_i,\varphi_i)\cong\mathcal W$ (by classification methods, or again, by hand). Again to prove genericity, we will show that $\mathcal K_0$ is a \Fraisse class, and $\mathcal W$ its \Fraisse limit.
\end{remark}

\begin{corollary} \label{cor:jep}
The classes $\mathcal{K}_{\mathcal W}$, $\mathcal{K}_0$, $\mathcal{K}_1$, and $\mathcal{K}_{\bar p}$, where $\bar p$ is a supernatural number of infinite type, have the JEP.
\end{corollary}

\begin{proof}
Since
\[
\Mor_{GR,1}\subseteq \Mor_{GR,\bar p}\subseteq\Mor_{GR,0}
\]
whenever $\bar p$ is a supernatural number of infinite type, it is enough to prove JEP for $\mathcal{K}_{\mathcal W}$ and $\mathcal{K}_1$.

For Razak blocks, let $(A_{n,k},\sigma),(A_{n',k'},\tau)\in\Obj_R$. Let $C=A_{nn',(nn'-1)kk'}$ and let $\lambda$ be the Lebesgue trace on $C$. Define
\[
\varphi_1=\psi_{nn',(nn'-1)k,k'}\circ\varphi_{n,k,n'}\colon A\to C
\] 
and 
\[
\varphi_2=\psi_{nn',(nn'-1)k',k}\circ\varphi_{n',k',n}\colon B\to C,
\]
where the maps $\varphi_{\cdot,\cdot,\cdot}$ and $\psi_{\cdot,\cdot,\cdot}$ refer to those constructed in Proposition~\ref{prop:bhishansmaps}. Let $\tilde\varphi_1=\varphi_1\circ\varphi_{\sigma\mapsto(\lambda\circ\varphi_1)}$ and $\tilde\varphi_2=\varphi_2\circ\varphi_{\tau\mapsto(\lambda\circ\varphi_2)}$, where $\varphi_{\tau\mapsto\sigma}$ is the transition map constructed in Proposition~\ref{prop:onetoanother}. Since all the maps used in Propositions~\ref{prop:bhishansmaps} and \ref{prop:onetoanother} satisfy the equivalent conditions of Proposition~\ref{prop:tpstarhoms}, $\tilde\varphi_1$ and $\tilde\varphi_2$ belong to $\Mor_R$; this shows that $\mathcal K_{\mathcal W}$ has the JEP. 

For generalised blocks, let $(B_{n,k},\sigma), (B_{m,l},\tau)\in\Obj_{GR}$. By Proposition~\ref{prop:genmaps}\eqref{genmaps3} with $k'=ml(ml-1)$ and $k'=nk(nk-1)$ respectively, there are maps $\varphi_1\colon B_{n,k}\to B_{nk,(nk-1)ml(ml-1)}$ and $\psi_1\colon B_{m,l}\to B_{ml,(ml-1)nk(nk-1)}$ with trivial $K$-theory. Another application of Proposition~\ref{prop:genmaps}\eqref{genmaps3} with $k'=2$ gives maps
\[
\varphi_2\colon B_{nk,(nk-1)ml(ml-1)}\to B_{nk(nk-1)ml(ml-1), 2(nk(nk-1)ml(ml-1)-1)}
\]
and 
\[
\psi_2\colon B_{ml,(ml-1)nk(nk-1)}\to B_{ml(ml-1)nk(nk-1), 2(ml(ml-1)nk(nk-1)-1)},
\]
again with trivial $K$-theory. All maps involved satisfy the equivalent conditions of Proposition~\ref{prop:tpstarhoms}, and therefore belong to $\Mor_{GR,1}$; hence the maps $\varphi_2\circ\varphi_1\circ\varphi_{\sigma\mapsto(\lambda\circ\varphi_2\circ\varphi_1)}$ and $\psi_2\circ\psi_1\circ\varphi_{\tau\mapsto(\lambda\circ\psi_2\circ\psi_1)}$ witness the JEP.
\end{proof}

\section{Distances} \label{S.distances}

We define and study several distances between $^*$-homomorphisms, measures, and diagonal maps.

\subsection{Distances between $^*$-homomorphisms} \label{subsection:distances}
\sloppy
Let $A$ and $B$ be \cstar-algebras, let $G\subseteq A$ be compact and let $\varepsilon>0$. For $^*$-homomorphisms $\varphi,\psi\colon A\to B$ we define the \emph{unitary distance relative to $G$} between $\varphi$ and $\psi$ as
\[
d_\mathcal{U}^G(\varphi,\psi) = \inf_{u\in\mathcal{U}(\tilde B)}\sup_{f\in G}\norm{\varphi(f)-u\psi(f)u^*},
\]
where $\mathcal U(\tilde B)$ is the unitary group of the unitisation of $B$. When $G$ is a ``separating family'', for example, if $G$ equals the set of $1$-Lipschitz contractions in a generalised Razak block, this gives a meaningful notion of distance between approximate unitary equivalence classes of $^*$-homomorphisms. If $B$ is finite dimensional and $f\in A$ is positive, the unitary distance $d_{\mathcal U}^{\{f\}}(\varphi,\psi)$ equals the optimal matching distance between the eigenvalues of $\varphi(f)$ and $\psi(f)$.
\fussy

 Another important distance relates diagonal maps. Let $A\subseteq C([0,1],M_n)$ and $B\subseteq C([0,1],M_m)$, and let $\varphi,\psi\colon A\to B$ be diagonal maps with associated $\{\xi_i^\varphi\}_{i\leq j}$ and $\{\xi_i^\psi\}_{i\leq j}$. The \emph{diagonal distance} between $\varphi$ and $\psi$ is defined as
 \[
 d_{\partial}(\varphi,\psi)=\sup_{t\in[0,1]}\sup_i|\xi_i^\varphi-\xi_i^\psi|.
 \]
 \begin{lemma}\label{L.dist1}
Let $A$ and $B$ be (generalised) Razak blocks. Let $G\subseteq A$ be a set of $L$-Lipschitz functions. Let $\varphi,\psi\colon A \to B$ be diagonal maps. Then 
\[
\sup_{t\in [0,1]}d_{\mathcal U}^G(\pi_{\varphi,t},\pi_{\psi,t})\leq L\cdot d_{\partial}(\varphi,\psi).
\]
Moreover, if $A$ and $B$ are Razak blocks, then
\[
\sup_{t\in (0,1)\cup\{\infty\}}d_{\mathcal U}^G(\pi_{\varphi,t},\pi_{\psi,t})=\sup_{t\in [0,1]}d_{\mathcal U}^G(\pi_{\varphi,t},\pi_{\psi,t}).
\]
\end{lemma}
\begin{proof}
Let $m$ be such that $B\subseteq C([0,1],M_m)$ (that is, $m=nk$ if $B=A_{n,k}$ and $m=2nk$ if $B=B_{n,k}$). Let $\{\xi_i^\varphi\}$ and $\{\xi_i^\psi\}$ be the continuous maps associated to $\varphi$ and $\psi$, so that for all $t\in [0,1]$ we have
\[
\pi_{\varphi,t}\sim_u\diag(\pi_{\xi_1^\varphi(t)},\ldots,\pi_{\xi_j^\varphi(t)}) \text{ and }\pi_{\psi,t}\sim_u\diag(\pi_{\xi_1^\psi(t)},\ldots,\pi_{\xi_j^\psi(t)}). 
\]
Then 
\begin{align*}
d_{\mathcal U}^G(\pi_{\varphi,t},\pi_{\psi,t})
&\leq \sup_{f\in G}\sup_i\norm{f(\xi_i^\varphi(t))-f(\xi_i^\psi(t))}\\
&\leq \sup_{f\in G}\sup_i L\cdot |\xi_i^\varphi(t)-\xi_i^\psi(t)|\leq L\cdot d_{\partial}(\varphi,\psi),
\end{align*}
where the second to last inequality follows from the assumption that all elements of $G$ are $L$-Lipschitz.

The second statement follows from the fact that for a Razak block the space of representations is Hausdorff when endowed with the hull-kernel topology. Since $\pi_{\varphi,0}=\diag(\underbrace{\pi_{\varphi,\infty}}_n)$, if $\pi_{\varphi,0}\sim_u\pi_{\psi,0}$, then $\pi_{\varphi,\infty}\sim_u\pi_{\psi,\infty}$. Quantifying this, we get 
\[
d_{\mathcal U}^G(\pi_{\varphi,0},\pi_{\psi,0})=d_{\mathcal U}^G(\pi_{\varphi,\infty},\pi_{\psi,\infty}).\qedhere
\]
\end{proof}
\begin{remark}\label{Rem.GenBad}
The second part of Lemma~\ref{L.dist1} does not hold for generalised Razak blocks, as the hull-kernel topology is not Hausdorff, and it is not true that if $\pi_{\varphi,0}\sim_u\pi_{\psi,0}$ for all $t$ then $\pi_{\varphi,\infty_1}\sim_u\pi_{\psi,\infty_1}$. For example, consider the identity map on $B_{n,k}$ and let $\varphi$ be the map obtained by swapping $a_f$ and $b_f$ (e.g., Corollary~\ref{cor:Ktheoryreverting}). Then for all $G\subseteq B_{n,k}$ and $t\in [0,1]$ we have that 
\[
d_{\mathcal U}^G(\pi_{Id,t},\pi_{\psi,t})=0,
\]
 but for every $f\in B_{n,k}$ such that $a_f=-(1_k)$ and $b_f=1_k$ we have that 
\[
d_{\mathcal U}^{\{f\}}(\pi_{Id,\infty_1},\pi_{\varphi,\infty_1})=d_{\mathcal U}^{\{f\}}(\pi_{Id,\infty_2},\pi_{\varphi,\infty_2})=2.
\]
One immediately notices that the maps of Remark~\ref{Rem.GenBad} have different $K$-theory.
\end{remark}

The following shows that, for maps with small diameter, $d_\partial$ can be controlled by traces.

\begin{lemma}\label{lemma:rewriting2}
Let $A$ and $B$ be (generalised) Razak blocks. Let $\sigma\in T_{fd}(A)$ and $\tau\in T_{fd}(B)$, and suppose that $\varphi,\psi\colon (A,\sigma)\to (B,\tau)$ are diagonal maps with $\partial(\varphi),\partial(\psi)<\varepsilon$. Then
\[
d_\partial(\varphi,\psi)<3\varepsilon.
\]
\end{lemma}
\begin{proof}
Let $\{\xi_i^\varphi\}_{i\leq l}$ and $\{\xi_i^\psi\}_{i\leq l}$ the continuous maps associated to $\varphi$ and $\psi$ respectively. Suppose that there are $i\leq l$ and $t\in[0,1]$ such that $\xi_i^\varphi(t)+3\varepsilon<\xi_i^\psi(t)$. 
Let $c=\max\xi_i^\varphi$ and $d=\min\xi^\psi_{i}$. Since $\varphi$ and $\psi$ both have diameter $<\varepsilon$, then $d-c>\varepsilon$. Let $c'=c+\varepsilon/2$. Notice that if $j\leq i$, then the image of $\xi_{j}^\varphi$ is included in $[0,c]$, and if $i\leq j$, then the image of $\xi^\psi_j$ is contained in $[c',1]$. Since $\sigma=\varphi^*(\tau)$, then 
\[
\frac{i}{l}=\sum_{j\leq i}\mu_\tau((\xi^\varphi_j)^{-1}([0,1]))\leq \mu_\sigma([0,c]),
\]
 and since $\sigma=\psi^*(\tau)$, then 
 \[
\mu_\sigma([0,c'))\leq \sum_{j< i}\mu_\tau((\xi^\psi_j)^{-1}([0,1]))=\frac{i-1}{l}.
 \]
Since $c<c'$, this is a contradiction.
\end{proof}

\subsection{Measures}
Let $(X,d)$ be a separable metric space. Let $\mathcal{M}(X)$ denote the space of Borel probability measures on $X$, let $\mathcal{M}_f(X)$ denote those measures in $\mathcal{M}(X)$ that are faithful, and let $\mathcal{M}_{fd}(X)$ denote those that are faithful and diffuse.

There are many distances that provide a metrisation of the $w^*$-topology on $\mathcal{M}(X)$, such as the Wasserstein metric, and the L\'evy-Prokhorov metric (see e.g., \cite[\S2]{Jacelon:2021wa}).
Most useful in the context of \cstar-algebras is the \emph{optimal matching distance} (or \emph{bottleneck distance})
\[
\bb(\mu,\nu) = \sup_{U\subseteq X\,\text{open}}\inf\{r>0 \mid \mu(U) \le \nu(U_r) \textrm{ and } \nu(U) \le \mu(U_r)\},
\]
where $U_r=\{x\in U\mid d(U,x)<r\}$. Notice that for $X=[0,1]$, it is enough to quantify over open intervals (see e.g., the proof of \cite[Theorem 2.1]{Hiai:1989aa}). Moreover, when restricted to faithful, diffuse measures, $\bb$ is also a metrisation of the $w^*$-topology. 

Recall that if $A$ is a (generalised) Razak block and $\sigma,\tau\in T_{fd}(A)$, then $\varphi_{\sigma\mapsto\tau}$ denotes the transition map $(A,\sigma)\to (A,\tau)$ of Proposition~\ref{prop:onetoanother}. The following is a consequence of \cite[Proposition 2.2]{Jacelon:2021wa}.
\begin{proposition}\label{prop:measuring}
Let $A$ be a (generalised) Razak block, and let $\sigma,\tau\in T_{fd}(A)$. Then $
d_\partial(Id,\varphi_{\sigma\mapsto\tau})\le\bb(\mu_\sigma,\mu_\tau).$\qed
\end{proposition}

We now link our measure distance to diagonal maps. Fix $n,k\in\NN$. Notice that the maps $\varphi_{n,k,p}$ and $\varphi_{n,k,p,j}$ from either Proposition~\ref{prop:bhishansmaps}\eqref{emb:c0} or \ref{prop:genmaps}\eqref{genmaps1} have the same associated continuous maps (even though the map $\varphi_{n,k,p,j}$ only makes sense if $p$ is odd). Therefore, the pullback trace of the Lebesgue trace $\lambda$ is the same one. Let $\mu_\lambda$ be the Lebesgue measure associated to the Lebesgue trace $\lambda$.

\begin{proposition} \label{prop:measuring2}
Let $n,k\in\NN$. Let $\mu_p$ be the Borel probability measure on $[0,1)$ associated to the trace $\lambda_p=\lambda\circ\varphi_{n,k,p}$. Then $\mathfrak b(\mu_p,\mu_\lambda)\to 0$ as $p\to\infty$.
\end{proposition}

\begin{proof}
Let $\xi_1,\ldots,\xi_{p(pn-1)}$ be the maps associated to $\varphi_{n,k,p}$. We will show that for every interval $U$ we have $\mu_p(U)\leq\mu_\lambda(U_{\frac{3}{p}})$ and $\mu_\lambda(U)\leq\mu_p(U_{\frac{3}{p}})$, so that $\mathfrak b(\mu_p,\mu_\lambda)\leq\frac{3}{p}$. Let $j=|\{m\mid \frac{m}{p}\in U\}|$. Then, $\frac{j-1}{p}\leq \mu_\lambda(U)\leq \frac{j+1}{p}$. Moreover, either $U_{\frac{3}{p}}=[0,1]$, in which case we are done, or $[0,1]\setminus U$ contains an interval of length $\geq\frac{3}{p}$, in which case $\mu_\lambda(U_{\frac{3}{p}})\geq\frac{j-1}{p}+\frac{3}{p}=\frac{j+2}{p}$. Recall that 
\[
\mu_p(U)=\frac{1}{p(pn-1)}\sum_{i\leq p(pn-1)}\mu_\lambda(\{\xi_i^{-1}[U]\}),
\]
and that each $\xi_i$ has diameter $\le \frac{1}{p}$. Hence, if $i$ is such that $d(\xi_i(1),U)>\frac{1}{p}$, then $\xi_i^{-1}[U]=\emptyset$. Since $|\{i\mid \xi_i(1)=\frac{m}{p}\}|=pn-1$ for all $m$ with $0<m\leq p$, we therefore have
\[
\mu_p(U)\leq \frac{j+2}{p}\leq\mu_\lambda(U_{\frac{3}{p}}).
\]
On the other hand, if $i$ is such that $d(\xi_i(1),U)\leq\frac{2}{p}$, then $\xi_i^{-1}[U_{\frac{3}{p}}]=[0,1]$. By our choice of $j$, there are at least $(j+2)(pn-1)$ such maps. Hence, 
\[
\mu_\lambda(U)\leq\frac{j+2}{p}\leq\mu_p(U_{\frac{3}{p}}).\qedhere
\]
\end{proof}

The next result aims to bring together our distances and their relations.
\begin{theorem}\label{T.smalldistance}
Let $A$, $B$, and $C$ be Razak blocks, and let $\sigma\in T_{fd}(A)$, $\tau_1\in T_{fd}(B)$ and $\tau_2\in T_{fd}(C)$. Let $\varphi_1\colon (A,\sigma)\to (B,\tau_1)$ and $\varphi_2\colon (A,\sigma)\to (C,\tau_2)$ be $^*$-homomorphisms, $G\subseteq A$ be finite, and $\varepsilon>0$. Then there is a Razak block $D$ and two $^*$-homomorphisms $\psi_1\colon (B,\tau_1)\to (D,\lambda)$ and $\psi_2\colon (C,\tau_2)\to (D,\lambda)$ such that 
\[
\sup_{t\in (0,1)\cup\{\infty\}}d_{\mathcal U}^G(\pi_{\psi_1\circ\varphi_1,t},\pi_{\psi_2\circ\varphi_2,t})<\varepsilon.
\]
\end{theorem}
\begin{proof}
Since $\mathcal K_{\mathcal W}$ has the JEP (Corollary~\ref{cor:jep}), and thanks to the existence of transition maps, we can assume that $B=C$ and that $\sigma=\tau_1=\tau_2=\lambda$, the latter being the Lebesgue trace. Furthermore, we can assume that $G$ consists of $1$-Lipschitz functions. Using Lemma~\ref{lemma:diameterfacts}, pick $\delta>0$ such that if $\psi$ is a map of diameter $<\delta$ then $\partial(\psi\circ\varphi_1),\partial(\psi\circ\varphi_2)<\varepsilon/3$. 

Say $B=A_{n,k}$. By Proposition~\ref{prop:measuring2}, we can find $p$ large enough such that, with $\varphi_{n,k,p}$ the map from Proposition~\ref{prop:bhishansmaps} and $\mu_p=\lambda\circ\varphi_{n,k,p}$, we have $\mathfrak b(\mu_p,\mu_\lambda) <\delta/2$, and so by Proposition~\ref{prop:measuring}, $d_\partial(Id,\varphi_{\lambda\mapsto\lambda_p})<\delta/2$. Let $\psi_1=\psi_2=\varphi_{n,k,p}\circ\varphi_{\lambda\mapsto\lambda_p}$. Let $D=A_{pn,(pn-1)k}$. Notice that
\[
\psi_1,\psi_2\colon (B,\lambda)\to (D,\lambda), 
\]
and therefore 
\[
\psi_1\circ\varphi,\psi_2\circ\varphi_2\colon (A,\lambda)\to (D,\lambda).
\]
Since $d_\partial(Id,\varphi_{\lambda\mapsto\lambda_p})<\delta/2$ and $\partial(\varphi_{n,k,p})<\delta/2$, then $\partial(\psi_1)=\partial(\psi_2)<\delta$. By our choice of $\delta$ we then have that
\[
\partial(\psi_1\circ\varphi_1),\partial(\psi_2\circ\varphi_2)<\varepsilon/3.
\]
Applying Lemma~\ref{lemma:rewriting2} with $\sigma=\tau=\lambda$, we obtain $d_{\partial}(\psi_1\circ\varphi_1,\psi_2\circ\varphi_2)<\varepsilon$. The thesis follows from Lemma~\ref{L.dist1}.
\end{proof}
\subsection{Generalised Razak blocks}
Trying to reproduce the proof of Theorem~\ref{T.smalldistance} verbatim for generalised Razak blocks only gives that, once the appropriate morphisms are given,
\[
\sup_{t\in [0,1]}d_{\mathcal U}^G(\pi_{\psi_1\circ\varphi_1,t},\pi_{\psi_2\circ\varphi_2,t})<\varepsilon.
\]
 By Remark~\ref{Rem.GenBad}, this is not enough to ensure that the unitary orbits of \emph{all} irreducible representations of $\psi_1\circ\varphi_2$ and of $\psi_2\circ\varphi_2$ are close to each other. 
To obtain an appropriate version of Theorem~\ref{T.smalldistance}, we then need to take $K$-theory into account.
\begin{theorem}\label{T.gensmalldistance}
Let $\bar p$ be a supernatural number of infinite type. 
Let $A$, $B$, and $C$ be generalised Razak blocks, and let $\sigma\in T_{fd}(A)$, $\tau_1\in T_{fd}(B)$ and $\tau_2\in T_{fd}(C)$. Let $\varphi_1\colon (A,\sigma)\to (B,\tau_1)$ and $\varphi_2\colon (A,\sigma)\to (C,\tau_2)$ be $^*$-homomorphisms whose $K$-theory divides $\bar p$. Let $G\subseteq A$ be finite, and $\varepsilon>0$. 
 Then there is a generalised Razak block $D$ and two $^*$-homomorphisms $\psi_1\colon (B,\tau_1)\to (D,\lambda)$ and $\psi_2\colon (C,\tau_2)\to (D,\lambda)$ such that the $K$-theory of $\psi_1\circ\varphi_1$ and of $\psi_2\circ\varphi_2$ both divide $\bar p$, and
\[
\sup_{t\in [0,1]\cup\{\infty_1,\infty_2\}}d_{\mathcal U}^G(\pi_{\psi_1\circ\varphi_1,t},\pi_{\psi_2\circ\varphi_2,t})<\varepsilon.
\]
\end{theorem}
The rest of the section is dedicated to the proof of Theorem~\ref{T.gensmalldistance}. 

If $F$ and $H$ are two multisubsets (i.e., whose elements are counted with multiplicity) of $[0,1]$ of equal size, the optimal matching distance between the finitely supported counting measures $\mu_F$ and $\mu_H$ coincides with the infimum over all bijections $\sigma\colon F\to H$ of $\sup_{f\in F}|\sigma(f)-f|$. Ordering $F$ and $H$ as $F=\{f_i\}_{i\leq j}$ and $H=\{h_i\}_{i\leq j}$ (where $f_i\leq f_{i+1}$, and $h_i\leq h_{i+1}$ for all $i$), this in turn coincide with the quantity $\sup_i|f_i-h_i|$. We abuse notation and write $\mathfrak b(F,H)$ for $\mathfrak b(\mu_F,\mu_H)$.
\begin{defin}\label{defin:multisets}
Let $\ell\in\NN$ and let $F$ and $H$ be two multisubsets (i.e., whose elements are counted with multiplicity) of $[0,1]$ of equal size. Define 
\begin{equation*}
\mathfrak b_\ell(F,H)=\sup_{F'\subseteq F, H'\subseteq H,\\|F'|=|H'|\leq\ell}\mathfrak b(F\setminus F',H\setminus H').
\end{equation*}
\end{defin}
For two finite multisets, being $<\varepsilon$ in the distance $\mathfrak b_\ell$ corresponds to the fact that the two sets are so close that it doesn't matter if one slightly modifies them (by removing up to $\ell$ many elements), in that one is always able to match the elements of the remaining multisets up to $\varepsilon$.

\begin{lemma}\label{lemma:littlecounting}
Fix $\ell\in\NN$ and $\varepsilon>0$. Let $F$ and $H$ be finite multisets with the same size. Suppose that $\mathfrak b(F,H)<\varepsilon$. Suppose moreover $|F\cap U|, |H\cap U|\geq\ell$ whenever $U$ is an open interval of diameter $\geq\varepsilon$, when $F$ and $H$ are considered as multisets. Then $\mathfrak b_\ell(F,H)\leq 3\varepsilon$.
\end{lemma}
\begin{proof}
Say $|F|=|H|=j$. Order $F$ and $H$ as $F=\{f_1,\ldots,f_j\}$ and $H=\{h_1,\ldots,h_j\}$, where $f_i\leq f_{i+1}$ and $h_i\leq h_{i+1}$ for all $i\leq j$. By the paragraph preceding Definition~\ref{defin:multisets}, the bijection mapping $f_i$ to $h_i$ witnesses that $\mathfrak b(F,H)<\varepsilon$, hence $|f_i-h_i|<\varepsilon$ for all $i$. By the hypothesis, we have that $|f_i-f_{i+\ell}|\leq\varepsilon$ for all $i$, and similarly $|h_i-h_{i+\ell}|\leq\varepsilon$. Fix sets $F'\subseteq F$ and $H'\subseteq H$ of size $k$ with $k\leq\ell$. Write $F\setminus F'=\{f_1',\ldots,f_{j-k}'\}$ and $H\setminus H'=\{h'_1,\ldots,h'_{j-k}\}$ in increasing order. Then for all $i$ we have that $f_i\leq f_i'\leq f_{i+\ell}$, and similarly $h_i\leq h'_i\leq h_{i+\ell}$. In particular $|f_i-f_i'|\leq\varepsilon$ and equally $|h_i-h_i'|\leq\varepsilon$. Hence
\[
|f_i'-h_i'|\leq |f_i'-f_i|+|f_i-h_i|+|h_i-h_i'|\leq3\varepsilon.\qedhere
\]
\end{proof}

\begin{proof}[Proof of Theorem~\ref{T.gensmalldistance}]
Since $\mathcal K_{\bar p}$ has the JEP (Corollary~\ref{cor:jep}), we can assume that $\varphi_1$ and $\varphi_2$ have the same $K$-theory $\ell$, that $B=C=B_{n,k}$ where $n$ is even, and that $\sigma=\tau_1=\tau_2=\lambda$, $\lambda$ being the Lebesgue trace. Furthermore, we can suppose that all elements of $G$ are $1$-Lipschitz. By applying the maps $\varphi_{n,k,p,n/2}$ from Proposition~\ref{prop:genmaps} (which have trivial $K$-theories), we can suppose that $\partial(\varphi_1),\partial(\varphi_2)<\varepsilon/30$. Notice that this implies that for all $t\in [0,1]$, if we write the representation $\pi_{\varphi_1,t}$ as $u\diag(\pi_{s_{1,1}^t},\ldots,\pi_{s_{1,m}^t})u^*$, then for every open set $U$ of diameter $\geq\varepsilon/6$ there is $i$ such that $s_{1,i}^t\in U$. The same statement holds for the points $s_{2,i}^t$ associated to $\pi_{\varphi_2,t}$. Moreover, since $\partial(\varphi_1)<\varepsilon/30$, then for all $t,t'\in [0,1]$ we have that 
\[
\mathfrak b(\{s_{1,i}^t\}_{i\leq m},\{s_{1,i}^{t'}\}_{i\leq m})<\varepsilon/30,
\]
and similarly
\[
\mathfrak b(\{s_{2,i}^t\}_{i\leq m},\{s_{2,i}^{t'}\}_{i\leq m})<\varepsilon/30.
\]
Since $d_\partial(\varphi_1,\varphi_2)<\varepsilon/10$ (by Lemma~\ref{lemma:rewriting2}), then
\[
\mathfrak b(\{s_{1,i}^t\}_{i\leq m},\{s_{2,i}^t\}_{i\leq m})\leq\varepsilon/10,
\]
and therefore for all $t,t'\in [0,1]$ we have that
\[
\mathfrak b(\{s_{1,i}^t\}_{i\leq m},\{s_{2,i}^{t'}\}_{i\leq m})\leq\varepsilon/30+\varepsilon/30+\varepsilon/10=\varepsilon/6.
\]
\sloppy
Applying Lemma~\ref{lemma:KtheoryGen2} to the representations
\[
\rho_1=\diag(\underbrace{\pi_{\varphi_1,\infty_1}}_{n/2},\underbrace{\pi_{\varphi_1,\infty_2}}_{n/2-1})
\text{ and }
\rho_2=\diag(\underbrace{\pi_{\varphi_2,\infty_1}}_{n/2},\underbrace{\pi_{\varphi_2,\infty_2}}_{n/2-1}),
\]
and to the representations
\[
\rho_3=\diag(\underbrace{\pi_{\varphi_1,\infty_2}}_{n/2},\underbrace{\pi_{\varphi_1,\infty_1}}_{n/2-1})
\text{ and }
\rho_4=\diag(\underbrace{\pi_{\varphi_2,\infty_2}}_{n/2},\underbrace{\pi_{\varphi_2,\infty_1}}_{n/2-1}),
\]
noticing that these two pairs have the same $K$-theory, we can find $j\in\NN$ and points $x_1,\ldots,x_j,y_1,\ldots,y_j,w_1,\ldots,w_j,z_1,\ldots,z_j\in [0,1]$ such that $\diag(\rho_1,\pi_{x_1},\ldots,\pi_{x_j})$ and $\diag(\rho_2,\pi_{y_1},\ldots,\pi_{y_j})$ are unitarily equivalent and  $\diag(\rho_3,\pi_{w_1},\ldots,\pi_{w_j})$ and $\diag(\rho_4,\pi_{z_1},\ldots,\pi_{z_j})$ are unitarily equivalent.
\fussy

Let $p\geq 2j+1$ be odd, and let $\psi_1=\psi_2=\varphi_{n,k,p}\circ\varphi_{\lambda\mapsto\lambda_p}$, where $\lambda_p=\lambda\circ\varphi_{n,k,p}$ was defined in the statement of Proposition~\ref{prop:measuring2}. We claim that 
 \[
\sup_{t\in [0,1]\cup\{\infty_1,\infty_2\}}d_{\mathcal U}^G(\pi_{\psi_1\circ\varphi_1,t},\pi_{\psi_2\circ\varphi_2,t})<\varepsilon.
\]
First, since $\partial(\varphi_1),\partial(\varphi_2)\leq\varepsilon/9$, then $\partial(\psi_1\circ\varphi_1),\partial(\psi_2\circ\varphi_2)\leq\varepsilon/9$ (see Lemma~\ref{lemma:diameterfacts}). As $\psi_1\circ\varphi_1$ and $\psi_2\circ\varphi_2$ pull back the same trace, by Lemma~\ref{lemma:rewriting2} we have that $d_\partial(\psi_1\circ\varphi_1,\psi_2\circ\varphi_2)\leq\varepsilon/3$. As all elements of $G$ are $1$-Lipschitz, we get from Lemma~\ref{L.dist1} that
\[
\sup_{t\in [0,1]}d_{\mathcal U}^G(\pi_{\psi_1\circ\varphi_1,t},\pi_{\psi_2\circ\varphi_2,t})<\frac{\varepsilon}{3}.
\]
Consider now $\pi_{\psi_1\circ\varphi_1,\infty_1}$ and $\pi_{\psi_2\circ\varphi_2,\infty_1}$. By definition of $\varphi_{n,k,p}$, and since the transition map used to define $\psi_1$ (and $\psi_2$) does not affect the endpoints, we have that 
\[
\pi_{\psi_1\circ\varphi_1,\infty_1}=\diag(\rho_1,\pi_{\varphi_1,1/p},\pi_{\varphi_1,3/p},\ldots,\pi_{\varphi_1,(p-2)/p})
\]
and
\[
\pi_{\psi_2\circ\varphi_2,\infty_1}=\diag(\rho_2,\pi_{\varphi_2,2/p},\pi_{\varphi_2,4/p},\ldots,\pi_{\varphi_2,(p-1)/p}).
\]
Let $F=\{s_{1,i}^r\}_{i\leq m, r=1/p,\ldots,(p-2)/p}$ and $H=\{s_{2,i}^r\}_{i\leq m, r=2/p,\ldots,(p-1)/p}$, considered as multisets, so that 
\[
\pi_{\psi_1\circ\varphi_1,\infty_1}=\diag(\rho_1,\{\pi_t\}_{t\in F})\text{ and }\pi_{\psi_2\circ\varphi_2,\infty_1}=\diag(\rho_2,\{\pi_t\}_{t\in H}).
\] 
Notice that $\mathfrak b(F,H)\leq\varepsilon/6$. Recall moreover that for every $r$ and every open interval $U$ of diameter $\geq\varepsilon/6$, there is $i\leq m$ such that $s_{1,i}^{r}\in U$. Hence, for every such $U$, $|F\cap U|\geq (p-1)/2\geq j$, where $F$ is considered as a multiset. Similarly, $|H\cap U|\geq j$. Hence by Lemma~\ref{lemma:littlecounting}, $\mathfrak b_j(F,H)\leq\varepsilon/2$. Let us now look at the points $x_1,\ldots,x_j$ and $y_1,\ldots,y_j$. For every $i\leq j$, pick $t_i\in F$ such that $|x_i-t_i|<\varepsilon/6$, and  pick $h_i\in H$ such that $|y_i-h_i|\leq \varepsilon/6$. We pick these in such a way that (as multisets), $|\{t_i\}_{i\leq j}|=j=|\{h_i\}_{i\leq j}|$. Since 
\[
\mathfrak b(F\setminus \{t_i\}_{i\leq j},H\setminus \{h_i\}_{i\leq j})<\varepsilon/2
\]
and all elements of $G$ are $1$-Lipschitz, then 
\[
d_{\mathcal U}^G(\diag(\{\pi_t\}_{t\in F\setminus \{t_i\}}),\diag(\{\pi_t\}_{t\in H\setminus \{h_i\}}))<\varepsilon/2.
\]
By our choice of the points $t_i$ and $h_i$, we also have that 
\[
d_{\mathcal U}^G(\diag(\rho_1,\{\pi_t\}_{t\in \{t_i\}}),\diag(\rho_2,\{\pi_t\}_{t\in \{h_i\}})<\varepsilon/2.
\]
Bringing all of these together we get that 
\[
d_{\mathcal U}^G(\pi_{\psi_1\circ\varphi_1,\infty_1},\pi_{\psi_2\circ\varphi_2,\infty_1})<\varepsilon.
\]
The same exact calculation works for $\infty_2$, and therefore we have the thesis.
\end{proof}
The following will be used in the proceeding.
\begin{corollary}\label{cor:onemore}
Let $(A_i,\varphi_i)$ be the inductive sequence of Definition~\ref{defin:Z0}. Fix $i\leq j$ and let $G\subseteq A_i$ be finite, and $\varepsilon>0$. Suppose that $\psi_1,\psi_2\colon A_i\to A_j$ are such that 
\[
\sup_{t\in[0,1]}d_\mathcal U^G(\psi_1,\psi_2)<\varepsilon.
\]
Then
\[
\sup_{t\in[0,1]\cup\{\infty_1,\infty_2\}}d_\mathcal U^G(\varphi_{j}\circ\psi_1,\varphi_j\circ\psi_2)<\varepsilon.
\]
\end{corollary}
\begin{proof}
The proof follows by the argument of Theorem~\ref{T.gensmalldistance} and the fact that $p_j$ in the choice of the sequence $A_i$ (see Definition~\ref{defin:Z0}) is constructed using Lemma~\ref{lemma:KtheoryGen2} and Remark~\ref{remark:points}. In fact, the choice of $p$ in the proof of Theorem~\ref{T.gensmalldistance} does not depend on $G$ or $\varepsilon$, but only on the number $j$ of points needed to make the representations $\rho_1$ and $\rho_2$ ($\rho_3$ and $\rho_4$) unitarily equivalent.
\end{proof}

\section{Connecting unitaries and the main result}\label{S.proof}

The aim of this section is to connect the unitaries conjugating the point representations of two diagonal maps between (generalised) Razak blocks. Namely, let $A$ be a (generalised) Razak blocks, and suppose that $G\subseteq A$ is finite. The question is: If $B$ is a (generalised) Razak block and $\varphi$ and $\psi$ are diagonal maps $A\to B$, can we compute $d_{\mathcal U}^G(\varphi,\psi)$ in terms of $\sup_{t\in(0,1)\cup\{\infty_i\}_{i=1}^2}d_\mathcal{U}^G(\pi_{\varphi,t}, \pi_{\psi,t})$?

The following result, familiar to experts, shows that the above question has a positive answer if $A$ is of the form $C([0,1],M_n)$. The key ingredients of its proof are compactness of the interval, a strong form of path-connectedness of the group of unitary matrices entailed by the continuous functional calculus, and the fact that the algebraic $K_1$ group of $[0,1]$ is trivial. The original argument can be traced back to Thomsen (\cite{Thomsen:1992qf}).

\sloppy
\begin{proposition}[Thomsen \cite{Thomsen:1992qf}] \label{prop:thomsen}
Let $n,k\in\mathbb{N}$, let $A=C([0,1],M_n)$ and let $B$ be the one-dimensional NCCW complex $B= A(E,M_m,\alpha_0,\alpha_1)$ for some finite-dimensional \cstar-algebra $E=\bigoplus_{i=1}^{p}M_{k_i}$ and injective boundary maps $\alpha_0,\alpha_1\colon E\to M_m$. That is,
\[
B = \{f \in C([0,1],M_m) \mid f(0) = \alpha_0(a), f(1) = \alpha_1(a), a\in E\}.
\]
Let $G\subseteq A$ be compact. Then for any two diagonal $^*$-homomorphisms $\varphi,\psi\colon A\to B$,
\pushQED{\qed} 
\[
d_\mathcal{U}^G(\varphi,\psi) \le \sup_{t\in(0,1)\cup\{\infty_i\}_{i=1}^p}d_\mathcal{U}^G(\pi_{\varphi,t}, \pi_{\psi,t}).\qedhere
\]
\popQED
\end{proposition}
\fussy

The aim of the remainder of the section is to prove a version of Thomsen's result for (generalised) Razak blocks. We use the combinatorial reduction of such a block $A$ to $C([0,1])$ as described in  \cite[\S5]{Robert:2010qy}. There, it is shown how to obtain a finite sequence $A=A_0,A_1,\ldots,A_r=C([0,1])$ (which we will call the \emph{Robert sequence} of $A$), where for each $i$, $A_{i}$ is related to $A_{i-1}$ by either
\begin{enumerate}[(i)]
\item $A_{i}=\tilde{A}_{i-1}$ (adding a unit) or
\item $\tilde{A}_{i}=A_{i-1}$ (removing a unit) or
\item $A_{i}\otimes\mathbb{K} \cong A_{i-1}\otimes\mathbb{K}$ (stable isomorphism).
\end{enumerate}
(Here $\mathbb K$ denotes the algebra of compact operators on a separable infinite-dimensional Hilbert space $\mathbb H$).

Moreover, a careful reading of \cite[\S5]{Robert:2010qy} indicates that each stable isomorphism is an adjustment by either
\begin{itemize}
\item inflation or deflation of one of the points at infinity or
\item adding or removing a row of zeros.
\end{itemize}

In both cases, the isomorphism $\theta\colon A_i\otimes\mathbb{K} \to A_{i-1}\otimes\mathbb{K}$ is of the form $\theta(f)=ufu^*$ for a suitable unitary $u\in\mathcal{U}(\mathbb{H})$ that in particular maps $A_i$ into a matrix algebra over $A_{i-1}$ (or the other way round).

For example, one sees from the proof of \cite[Proposition 5.2.2]{Robert:2010qy} that the last step for $A=B_{n,k}$ is the stable isomorphism between
\[
\{f\in C([0,1],M_2) \mid f(0)=\diag(a,0), f(1)=\diag(b,0), a,b\in\mathbb{C}\}
\]
and $C([0,1])$. Two steps prior is the stable isomorphism between
\[
\{f\in C([0,1],M_3) \mid f(0)=\diag(a,b), f(1)=\diag(0,0,b), a\in M_2,b\in\mathbb{C}\}
\]
and
\[
\{f\in C([0,1],M_2) \mid f(0)=\diag(a,b), f(1)=\diag(0,b), a,b\in\mathbb{C}\}.
\]

\sloppy
\begin{lemma} \label{lemma:reduction}
Let $A$ be a (generalised) Razak block. Let $B=A(E,M_m,\alpha_0,\alpha_1)$ be a one-dimensional NCCW complex as in Proposition~\ref{prop:thomsen}, and let $\varphi,\psi:A\to B$ be diagonal $^*$-homomorphisms. Let $G\subseteq A$ be finite. Then there is a natural number $N$, a one-dimensional NCCW complex $B'=A(E',M_{m'},\alpha'_0,\alpha'_1)$, a finite set $G'\subseteq C([0,1],M_N)$, diagonal $^*$-homomorphisms $\varphi',\psi':C([0,1],M_N)\to B'$ and an increasing function $h:(0,\infty)\to(0,\infty)$ depending only on $A$ such that $\lim_{\varepsilon\to0}h(\varepsilon)=0$ and 
\[
d^{G'}(\varphi',\psi') \le h(d^G(\varphi,\psi)) \,,\, d^{G}(\varphi,\psi) \le h(d^{G'}(\varphi',\psi')),
\]
where $d^F$ is either the uniform or pointwise unitary distance relative to $F$.
\end{lemma}
\fussy

\begin{proof}
Let $A=A_0,\ldots,A_r=C([0,1])$ be the Robert sequence of $A$. We will inductively verify that for each $i$, there is a natural number $N_i$, a finite set $G'_i\subseteq M_{N_i}(A_{i})$, a one-dimensional NCCW complex $B_i$ and diagonal $^*$-homomorphisms $\varphi_i,\psi_i:M_{N_i}(A_{i})\to B_i$ that satisfy the required property.

If $A_{i}=\tilde{A}_{i-1}$, set $N_i=1$, $G_i=G_{i-1}\cup\{1\}$, $B_i=\tilde{B}_{i-1}$ and $\varphi_i,\psi_i$ the unitisations of $\varphi_{i-1},\psi_{i-1}$.

If $\tilde{A}_{i}=A_{i-1}$, set $N_i=1$, $G_i=\{g-\pi_i(g)1 \mid g\in G_{i-1}\}$ (where $\pi_i:A_{i-1}\to\mathbb{C}$ is the canonical quotient map), $B_i=B_{i-1}$ and $\varphi_i,\psi_i$ the restrictions of the unital maps $\varphi_{i-1},\psi_{i-1}$ to $A_i$.

If $A_i$ is obtained from $A_{i-1}$ by removing a row of zeros or deflating a point at infinity, then there is an isomorphism $\theta_i\colon A_i\otimes\mathbb{K} \to A_{i-1}\otimes\mathbb{K}$ of the form $\theta_i(f)=u_ifu_i^*$ that maps $A_i$ into $A_{i-1}$. Choose $N_i$ such that $G_{i-1}\subseteq\theta_i(M_{N_i}(A_i))$, extend $\varphi_{i-1}$ and $\psi_{i-1}$ to diagonal $^*$-homomorphisms $M_{N_i}(A_{i-1})\to M_{N_i}(B_{i-1})$ and set $G_i=\theta_i^{-1}(G_{i-1})$, $B_i=M_{N_i}(B_{i-1})$ and $\varphi_i=\varphi_{i-1}\circ\theta_i$, $\psi_i=\psi_{i-1}\circ\theta_i$.

If $A_i$ is obtained from $A_{i-1}$ by adding a row of zeros or inflating a point at infinity, then there is an isomorphism $\theta_i\colon A_i\otimes\mathbb{K} \to A_{i-1}\otimes\mathbb{K}$ of the form $\theta_i(f)=u_ifu_i^*$ that maps $A_i$ into some $M_{N_i}(A_{i-1})$ (and whose inverse maps $A_{i-1}$ into $A_i$). Set $G_i=\theta_i^{-1}(G_{i-1})$, $B_i=M_{N_i}(B_{i-1})$ and $\varphi_i=\varphi_{i-1}\circ\theta_i$, $\psi_i=\psi_{i-1}\circ\theta_i$ (again extending $\varphi_{i-1}$ and $\psi_{i-1}$ to diagonal $^*$-homomorphisms $M_{N_i}(A_{i-1})\to M_{N_i}(B_{i-1})$).

In the last two cases, since we are passing to a larger matrix algebra, the (pointwise or uniform) unitary distance is \emph{a priori} smaller, that is,
\[
d^{G_i}(\varphi_i,\psi_i) \le d^{G_{i-1}}(\varphi_{i-1},\psi_{i-1}).
\]
On the other hand, from the proof of \cite[Proposition 2.3.1(i)]{Robert:2010qy} (which shows how unitary conjugation descends to hereditary subalgebras in the stable rank one setting), there is a function $f$ with $\lim_{\varepsilon\to0}f(\varepsilon)=0$ such that
\[
d^{G_{i-1}}(\varphi_{i-1},\psi_{i-1}) \le f(d^{G_i}(\varphi_i,\psi_i)).
\]
The function $\max\{Id,f\}$ is as desired, and the function $h$ that we obtain at the end of the induction depends only on the number of stable isomorphisms entailed by the Robert sequence.
\end{proof}

The following is immediate from Proposition~\ref{prop:thomsen} and Lemma~\ref{lemma:reduction}.

\begin{corollary} \label{cor:generalthomsen}
Let $A$ and $B$ be (generalised) Razak blocks. Then there is a function $h\colon(0,\infty)\to(0,\infty)$ depending only on $A$ with $\lim_{\varepsilon\to0}h(\varepsilon)=0$ such that for any finite set $G\subseteq A$ and diagonal $^*$-homomorphisms $\varphi,\psi\colon A\to B$,
\pushQED{\qed} 
\[
d_\mathcal{U}^G(\varphi,\psi) \le h\left(\sup_{t\in(0,1)\cup\{\infty_i\}_{i=1}^2}d_\mathcal{U}^G(\pi_{\varphi,t}, \pi_{\psi,t})\right).\qedhere
\]
\popQED
\end{corollary}
\begin{remark}
Our choice of proving Corollary~\ref{cor:generalthomsen} by using Robert's reduction method is purely aesthetic. In fact, one could take a `by-hand' approach, similar to Masumoto's one. Say $A$ and $B$ are (generalised) Razak blocks, and that $\varphi,\psi\colon A\to B$ are diagonal maps sending the Lebesgue trace to the Lebesgue trace. One first shows that diagonal maps are close (in the point-norm topology) to maps whose associated unitaries are continuous (similarly to \cite[Proposition 3.5]{Masumoto.FraisseZ}). More than that, one then shows that the associated unitaries $u$ and $v$ can have a very standard form, and finds $\varphi',\psi'\colon A \to B$ such that, on a finite set $G$, $\varphi'$ is close to $\varphi$ and $\psi'$ is close to $\psi$, and with the property that $uv^*\in\tilde B$. This is similar to what was done for $\mathcal Z$ in \cite[\S4]{Masumoto.FraisseZ}. The proof then follows by composing the map $\psi$ with $Ad(uv^*)$.

Since this step, particularly for generalised Razak blocks, becomes extremely technical and fairly unpleasant to read, we decided to take the path offered by Robert's reduction.
\end{remark}

Corollary~\ref{cor:generalthomsen} is the key to conclude the proof of Theorem~\ref{thmi:main}.
\begin{theorem} \label{thm:nap}
The classes $\mathcal{K}_{\mathcal W}$, $\mathcal{K}_0$, $\mathcal{K}_1$, and $\mathcal{K}_{\bar p}$, where $\bar p$ is a supernatural number of infinite type, are \Fraisse classes.
\end{theorem}

\begin{proof}
Proposition~\ref{prop:WPPCCP} gives the WPP and the CCP, while by Corollary~\ref{cor:jep}, these classes have the JEP. Therefore it is enough to show such classes have the NAP. Since the function $h$ in Corollary~\ref{cor:generalthomsen} depends only on the domain algebra $A$, the result for Razak blocks follows directly from Theorem~\ref{T.smalldistance} and Corollary~\ref{cor:generalthomsen}, while that $\mathcal K_0$, $\mathcal K_1$ and $\mathcal K_{\bar p}$ are \Fraisse classes follows by applying Theorem~\ref{T.gensmalldistance} and Corollary~\ref{cor:generalthomsen}.
\end{proof}
We are ready to prove Theorem~\ref{thmi:main}, which we recall for convenience.
\begin{theorem}
The algebra $\mathcal W$ is the \Fraisse limit of the class $\mathcal K_\mathcal{W}$. The algebra $\mathcal Z_0$ is the \Fraisse limit of the class $\mathcal K_1$. If $\bar p$ is a supernatural number of infinite type, the algebra $\mathcal Z_0\otimes M_{\bar p}$ is the \Fraisse limit of the \Fraisse class $\mathcal K_{\bar p}$.
\end{theorem}
\begin{proof}
The approach for $\mathcal W$, $\mathcal Z_0$ and tensor products of the form $\mathcal Z_0\otimes M_{\bar p}$ is the same: we show that the sequence defining $\mathcal W$ (Definition~\ref{defin:W}), $\mathcal Z_0$ (Definition~\ref{defin:Z0}) and tensor products of the form $\mathcal Z_0\otimes M_{\bar p}$ (Remark~\ref{remark:Z0UHF}) satisfy the hypotheses of Theorem~\ref{thm:generic} for their respective classes. We only give the details for $\mathcal W$: the proof in the other cases follows exactly the same way.

Consider the sequence $A_i$ given by Definition~\ref{defin:W}, with $\varphi_i\colon A_i\to A_{i+1}$, so that $\mathcal W=\lim(A_i,\varphi_i)$. Let $\tau$ be the unique trace of $\mathcal W$, and $\tau_i$ be the faithful diffuse trace on $A_i$ which is the pullback of $\tau$ via the embedding $\varphi_{i,\infty}=\lim_{j>i}\varphi_{i,j}$. Notice that $\varphi_{i,j}\colon (A_i,\tau_i)\to (A_j,\tau_j)$.

We aim to show that such a sequence satisfies the conditions of Theorem~\ref{thm:generic}. For the first condition, notice that $A_i=A_{n_i,k_i}$ where $n_i=(i-1)!$. Therefore the maps in Proposition~\ref{prop:bhishansmaps}\eqref{emb:c1} give that each Razak block with an associated faithful diffuse trace $(A,\sigma)$ can be embedded in $(A_i,\tau_i)$ (for some $i$) in a trace preserving way. We are left with the second condition. Fix $i\in\NN$ and consider a trace preserving map $\psi\colon(A_i,\tau_i)\to (B,\tau)$ where $B$ is a Razak block and $\tau\in T_{fd}(B)$. Fix a finite set $F\subset A_i$, and let $\varepsilon>0$. Without loss of generality we can assume all functions in $F$ are $1$-Lipschitz.

 We can find $k$ large enough so that there is a trace preserving $\psi'\colon (B,\tau)\to(A_k,\tau_k)$. Let $h$ be the function given by Corollary~\ref{cor:generalthomsen} for $A_i$, and let $\delta>0$ such that $h(\delta)<\varepsilon$. Consider now $k'$ large enough so that both $\varphi_{i,k'}$ and $\varphi_{k,k'}\circ\psi'\circ\psi$ have diameter $<\delta/3$. Notice that both maps pull back $\tau_{k'}$ to $\tau_i$, hence by Lemma~\ref{lemma:rewriting2},
\[
d_\partial(\varphi_{i,k'},\varphi_{k,k'}\circ\psi'\circ\psi)<\delta.
\]
By applying Lemma~\ref{L.dist1}, we have that
\[
\sup_{t\in[0,1]}d_\mathcal U^G(\varphi_{i,k'},\varphi_{k,k'}\circ\psi'\circ\psi)<\delta, 
\]
hence
\[
\sup_{t\in [0,1]\cup\{\infty\}}d_\mathcal U^G(\varphi_{i,k'},\varphi_{k,k'}\circ\psi'\circ\psi)<\delta.
\]
By the choice of $h$, we have that there is $u\in \tilde A_{k'}$ such that
\[
\norm{\varphi_{i,k'}(f)-Ad(u)\circ\varphi_{k,k'}\circ\psi'\circ\psi (f)}<\varepsilon.
\]
The map $Ad(u)\circ\varphi_{k,k'}\circ\psi'$ gives the thesis.

In the case of $\mathcal Z_0$, one first gets a $k'$ large enough and a map $\psi'$ so that 
\[
\sup_{t\in[0,1]}d_\mathcal U^G(\varphi_{i,k'},\varphi_{k,k'}\circ\psi'\circ\psi)<\varepsilon.
\]
Then, using the definition of $p_i$ in Definition~\ref{defin:Z0}, Corollary~\ref{cor:onemore} implies that
\[
\sup_{t\in[0,1]\cup\{\infty_1,\infty_2\}}d_\mathcal U^G(\varphi_{i,k'+1},\varphi_{k,k'+1}\circ\psi'\circ\psi)<\varepsilon, 
\]
and therefore the thesis. The approach to $\mathcal Z_0\otimes M_{\bar p}$ is the same.
\end{proof}

The last class of study is the class $\mathcal K_0$, whose objects are generalised Razak blocks and such that there is no $K$-theory restriction on the maps between building blocks. One can show that the \Fraisse limit of $\mathcal K_0$ is monotracial, simple, and has $K_0$ equal to $\{0\}$. By classification, this limit must be isomorphic to $\mathcal W$. Another approach is to show that the class $\mathcal K_\mathcal W\cup\mathcal K_0$, whose objects are Razak blocks and generalised Razak blocks with an associated diffuse faithful trace, and maps are trace preserving $^*$-homomorphisms, is a \Fraisse class. WPP, CCP for the class $\mathcal K_\mathcal W\cup\mathcal K_0$ can be proved exactly as in Proposition~\ref{prop:WPPCCP}. For the JEP, notice that generalised Razak blocks can be viewed as subalgebras of Razak blocks, by rearranging the blocks via a permutation unitary (specifically, $B_{n,k}$ can be twisted to a subalgebra of $A_{n,2k}$), and that $A_{n,k}\oplus A_{n,k}$ can be viewed as a subalgebra of $B_{n,k}$. The class of Razak blocks is cofinal in the class whose objects are Razak and generalised Razak blocks, $\mathcal K_\mathcal W\cup\mathcal K_0$; therefore, the inductive sequence defining $\mathcal W$ is a generic sequence in this \Fraisse class.

\appendix
\section{Admissible maps}\label{App.Maps}
We conclude this section by showing that the technical definition of admissible embeddings is not needed for our applications of \Fraisse theory. Let $\mathcal K_{\mathcal Z}$ be the category whose objects are pairs $(Z_{p,q},\tau)$ where 
\[
Z_{p,q}=\{f\in C([0,1],M_{p}\otimes M_q)\mid f(0)\in 1\otimes M_q,\,f(1)\in M_p\otimes 1,\,\, p,q\text{ coprime }\}
\]
and $\tau$ is a faithful trace on $Z_{p,q}$. Let $\Mor_{\mathcal K_{\mathcal Z}}$ be the set of all morphisms $\varphi\colon Z_{p,q}\to Z_{p',q'}$ such that there are faithful traces $\sigma\in T(Z_{p,q})$ and $\tau\in T(Z_{p',q'})$ with $\sigma=\tau\circ\varphi$.

Theorems 3.5 and 3.13 in \cite{Masumoto.Real} showed that $\mathcal K_{\mathcal Z}$ is a \Fraisse class and that the Jiang-Su algebra $\mathcal Z$ is its limit. Masumoto then analyzed the structure of $\mathcal K_{\mathcal Z}$-admissible embeddings of $\mathcal Z$ into itself. The following completes his intuition.

\begin{lemma}\label{lemma:tracewitnesses}
Let $B$ be a one-dimensional NCCW complex as in Proposition\ref{prop:thomsen}, whose spectrum is Hausdorff (for example, $B=Z_{p,q}$ or $B=A_{n,k}$). Let $F\subset B$ be finite, $\varepsilon>0$ and $\sigma\in T_f(B)$. Then there is a finite set $G\subseteq B$ and a $\delta>0$ such that whenever $\tau\in T_{fd}(B)$ satisfies $|\tau(f)-\sigma(f)|<\delta$ for all $f\in G$, there is 
\[
\varphi\colon (B,\sigma)\to(B,\tau)
\]
such that
\[
\norm{\varphi(a)-a}<\varepsilon,\,\,a\in F.
\]
\end{lemma}
\begin{proof}
Given $\tau$, let $\varphi$ be the transition map $\varphi_{\sigma\mapsto\tau}$ of Proposition~\ref{prop:onetoanother}. Note that diffuseness of $\tau$ is enough to ensure continuity. It is known to experts, and straightforward to verify, that $\mathfrak{b}$ provides a metrisation of the $w^*$-topology on faithful measures on $[0,1]$. The thesis follows.
\end{proof}

\begin{theorem}\label{thm:alladmissibleZ}
Let $\psi\colon \mathcal Z\to \mathcal Z$ be a nonzero $^*$-homomorphism. Then $\varphi$ is $\mathcal K_{\mathcal Z}$-admissible.
\end{theorem}
\begin{proof}
Note that $\psi$ is unital and injective (as $\mathcal Z$ is projectionless and simple). Let $tr_\mathcal Z$ be the unique trace on $\mathcal Z$ and write $(\mathcal Z,tr_\mathcal Z)=\lim((Z_{p_i,q_i},\sigma_i),\varphi_i)$ with $\varphi_i\colon (Z_{p_i,q_i},\sigma_i)\to(Z_{p_{i+1},q_{i+1}},\sigma_{i+1})$ and $\sigma_i\in T_{fd}(Z_{p_i,q_i})$, making sure that every $n$ eventually divides $p_iq_i$. 
Let 
\[
\varphi_{i,\infty}\colon (Z_{p_i,q_i},\sigma_i)\to(\mathcal Z,tr_\mathcal Z)
\]
be defined as $\varphi_{i,\infty}=\lim_{j>i}\varphi_{i,j}$ where, for $i<j$, 
\[
\varphi_{i,j}=\varphi_{j-1}\circ\varphi_i.
\]

\begin{claim}
Let $\sigma\in T_f(Z_{p,q})$ and $\pi\colon (Z_{p,q},\sigma)\to(\mathcal Z,tr_{\mathcal Z})$. Let $F\subset Z_{p,q}$ be finite and let $\varepsilon>0$. Then there is a natural number $i$ and a trace preserving map
\[
\pi'\colon (Z_{p,q},\sigma)\to (Z_{p_i,q_i},\sigma_i)
\]
such that
\[
\norm{\varphi_{i,\infty}\circ\pi'(a)-\pi(a)}<\varepsilon, \,\, a\in F.
\]
In particular, $\pi$ is $\mathcal K_{\mathcal Z}$-admissible.
\end{claim}
\begin{proof}
We will assume $F$ is made of contractions. Obtain $G$ and $\delta<\varepsilon$ from Lemma~\ref{lemma:tracewitnesses} applied to $F$, $\varepsilon$ and $\sigma$. Since $Z_{p,q}$ is semiprojective, and by \cite[Theorem~3.1]{Blackadar.Shape} there is $i$ and a $^*$-homomorphism $\rho\colon Z_{p,q}\to Z_{p_i,q_i}$ such that 
\[
\norm{\varphi_{i,\infty}\circ\rho (a)-\pi(a)}<\delta/2, \,\, a\in G\cup F.
\]
We can also suppose that $pq$ divides $p_iq_i$ and therefore, by the first paragraph after the proof of \cite[Proposition~3.2]{Masumoto.Real}, $\rho$ is diagonal. Let $\tilde\sigma=\sigma_i\circ\rho$. If $\tilde\sigma$ is diffuse, let $\tilde{\tilde\sigma}=\tilde\sigma$. Otherwise, by twiddling the continuous functions associated to $\rho$ so that they are finite-to-one, we can construct $\rho'\colon Z_{p,q}\to Z_{p_i,q_i}$ such that $\norm{\rho(a)-\rho'(a)}<\delta/2$ for $a\in G\cup F$, and such that the map $\tilde{\tilde\sigma}$ defined as $\tilde{\tilde\sigma}=\sigma_i\circ\rho'$ is faithful and diffuse. Note that, for $a\in F\cup G$,
\[
|\sigma(a)-\tilde{\tilde\sigma}(a)|=|tr_\mathcal Z(\pi(a))-\sigma_i(\rho'(a))|\leq \delta/2+|tr_\mathcal Z(\pi(a))-\sigma_i(\rho(a))|.
\]
Since $\varphi_{i,\infty}\colon (Z_{p_i,q_i},\sigma_i)\to(\mathcal Z,tr_\mathcal Z)$ and $\norm{\varphi_{i,\infty}\circ\rho(a)-\pi(a)}<\delta/2$, we have that $|\sigma(a)-\tilde{\tilde\sigma}(a)|<\delta$. Applying Lemma~\ref{lemma:tracewitnesses}, we can find a $\psi\colon (Z_{p,q},\sigma)\to(Z_{p,q},\tilde{\tilde\sigma})$ with $\norm{\psi(a)-a}<\varepsilon$. Then, $\pi'=\psi\circ\rho'$ satisfies the thesis.
\end{proof}

Since every embedding $\varphi\colon (Z_{p,q},\sigma)\to(\mathcal Z,tr_\mathcal Z)$ is $\mathcal K_{\mathcal Z}$-admissible, then so is $\psi\restriction(\varphi_{i,\infty}(Z_{p_i,q_i}),tr_\mathcal Z)$. By Remark~\ref{remark:admissible}, this suffices.
\end{proof}

We now prove the correspondent of Theorem~\ref{thm:alladmissibleZ} for $\mathcal W$ and $\mathcal{Z}_0$. The proof is necessarily slightly different due to the absence of the unit, but the strategy is similar.  We appeal to classification machinery (none of which, we again stress, is needed in the main body of the article), namely the following consequence of \cite[Theorem 1.0.1, Proposition 6.1.1, Proposition 6.2.3]{Robert:2010qy}.

\begin{theorem}[Robert] \label{thm:robert}
Let $A$ and $B_i$, $i\in\mathbb{N}$, be one dimensional NCCW complexes with trivial $K_1$, and suppose that there are connecting maps $\varphi_i:B_i\to B_{i+1}$ such that $B=\lim(B_i,\varphi_i)$ is simple and has a unique trace $tr_B$. Then for every $\sigma\in T_{f}(A)$, there exists $\varphi:(A,\sigma)\to (B,tr_B)$. If $\psi:(A,\sigma)\to (B,tr_B)$ is another such map with $K_0(\varphi)=K_0(\psi)$, then $\varphi$ and $\psi$ are approximately unitarily equivalent.
\end{theorem}

In the cases of interest (that is, $B=\mathcal{W}$ or $B=\mathcal{Z}_0$), the existence part of Theorem~\ref{thm:robert} follows from local existence (Propositions~\ref{prop:bhishansmaps} and \ref{prop:genmaps}), local uniqueness (Theorems~\ref{T.smalldistance}, \ref{T.gensmalldistance} and Corollary~\ref{cor:generalthomsen}), and an intertwining argument. However, the uniqueness statement is not quite accessible by our results because we cannot ensure that the maps obtained by the application of \cite[Theorem~3.1]{Blackadar.Shape} are trace preserving for any traces.

\begin{theorem}\label{thm:alladmissibleW}
Let $\psi\colon \mathcal W\to\mathcal W$ be a trace preserving $^*$-homomorphism. Then $\psi$ is $\mathcal K_{\mathcal W}$-admissible.
\end{theorem}
\begin{proof}
Let $n,k\in\NN$, $\sigma\in T_{fd}(A_{n,k})$, $\pi\colon (A_{n,k},\sigma)\to(\mathcal W,tr)$ (where $tr_\mathcal{W}$ is the unique trace on $\mathcal W$). Let $F\subset A_{n,k}$ be finite and let $\varepsilon>0$. By Proposition~\ref{prop:bhishansmaps} and the uniqueness of $\mathcal{W}$ as the \Fraisse limit of the class $\mathcal{K}_\mathcal{W}$, there is a sequence $A_{n_i,k_i}$, for $i\in\mathbb{N}$, traces $\sigma_i\in T_{fd}(A_{n_i,k_i})$ and trace preserving maps 
\[
\varphi_i\colon (A_{n_i,k_i},\sigma_i)\to (A_{n_{i+1},k_{i+1}},\sigma_{i+1})
\]
such that
\begin{itemize}
\item $n=n_0$, $k=k_0$ and $\sigma=\sigma_0$
\item $\lim (A_{n_i,k_i},\varphi_i)=\mathcal W$
\item if $\varphi_{i,\infty}\colon A_{n_i,k_i}\to\mathcal W$ is defined as $a\mapsto\lim_{j\geq i}\varphi_{ij}(a)$, then 
\[
\varphi_{i,\infty}(a)\colon (A_{n_i,k_i},\sigma_i)\to (\mathcal W,tr_\mathcal{W}).
\]
\end{itemize}
By Theorem~\ref{thm:robert}, there is a unitary $u$ in the unitisation of $\mathcal{W}$ such that the map $\varphi$ defined as $\varphi=\Ad(u)\circ\varphi_{0,\infty}$ satisfies
\[
\varphi\colon (A_{n,k},\sigma)\to (\mathcal W,tr_\mathcal{W})
\]
and
\[
\norm{\varphi(a)-\pi(a)}<\varepsilon, \,\,\, a\in F.
\]
As above, thanks to Remark~\ref{remark:admissible}, this is sufficient.
\end{proof}

The following is obtained in exactly the same way.

\begin{theorem}\label{thm:alladmissibleZ0}
Let $\psi\colon \mathcal Z_0\to\mathcal Z_0$ be a trace preserving $^*$-homomorphism. Then $\psi$ is $\mathcal K_1$-admissible. Similarly, for any supernatural number $\bar p$ of infinite type, every trace preserving $^*$-homomorphism $\psi\colon \mathcal{Z}_0\otimes M_{\bar p} \to\mathcal{Z}_0\otimes M_{\bar p}$ is $\mathcal K_{\bar p}$-admissible. \qed
\end{theorem}

\end{document}